\documentclass[a4paper,11pt]{amsart}
\usepackage[utf8]{inputenc}
\usepackage{amssymb,amsmath,amsthm,mathtools,amsfonts}
\usepackage[margin=1in]{geometry}
\usepackage{dsfont}
\newtheorem{lemma}{Lemma}
\newtheorem{theorem}{Theorem}
\newtheorem*{remark}{Remark}
\newtheorem{proposition}{Proposition}
\title{The Fourier transform on $2$-step Lie groups.}
\author{Guillaume L\'{e}vy$^{1}$}
\address{$^{1}$Laboratoire Jacques-Louis Lions, UMR 7598, Université Pierre et Marie Curie, 75252 Paris Cedex 05, France.}
\email{$^{1}${levy@ljll.math.upmc.fr}}

\begin{document}
\begin{abstract}
In this paper, we investigate the behavior of the Fourier transform on finite dimensional $2$-step Lie groups and develop a general theory akin to that of the whole space or the torus.
We provide a familiar framework in which computations are made sensibly easier than with the usual representation-theoretic Fourier transform.
In addition, we study the 'singular frequencies' of the group, at which the canonical bilinear antisymmetric form degenerates.
We also exhibit a specific example for which partial degeneracy of the canonical form occurs, as opposed to the full degeneracy at the origin.
We thus extend the results from \cite{BahouriCheminDanchinHeisenberg}.
\end{abstract}
\maketitle
\section{Introduction}

\subsection{Definition of $2$-steps Lie groups}

Let $G$ be a real, simply connected nilpotent Lie group.
Denote by $g$ its Lie algebra.
It is well-known (see \cite{Hebisch}) that for such groups, the exponential map 
$$\exp : g \to G$$ 
is a global diffeomorphism from $g$ onto $G$.
This map becomes a Lie isomorphism once one endows the Lie algebra $g$ with the group law given by the Baker-Campbell-Hausdorff formula, which terminates after a finite number of terms since $G$ is nilpotent.
In order to lighten the notation, we will henceforth assume that $G$ is the set $\mathbb{R}^n$ endowed with some group law.

In the sequel, we will restrict our attention to nilpotent groups of step $2$, for which all commutators are central.
That is, we assume that for any $x,y,z \in \mathbb{R}^n$, we have $[x,[y,z]] = 0$.
Let us denote by $p$ the dimension of the center of the group.
Then, there exists an integer $m$, a decomposition $\mathbb{R}^n = \mathbb{R}^m \oplus \mathbb{R}^p$ and a bilinear, antisymmetric map 
$$\sigma : \mathbb{R}^m \times \mathbb{R}^m \to \mathbb{R}^p$$ 
such that, for $Z,Z' \in \mathbb{R}^m$ and $s,s' \in \mathbb{R}^p$,
\begin{equation}
 (Z,s)\cdot (Z',s') = (Z+Z', s+s' + \frac 12 \sigma(Z,Z')).
 \label{DefLoiGroupe}
\end{equation}
The map $\sigma$ and the integers $m,p$ are determined by the group law and dimension. 
Conversely, for any integers $m,p$ and any bilinear, antisymmetric map $\sigma : \mathbb{R}^m \times \mathbb{R}^m \to \mathbb{R}^p$, 
one may define a Lie group of step $2$ by the formula $\eqref{DefLoiGroupe}$.
Now, given $\lambda \in \mathbb{R}^p$, we define the matrix $\mathcal{U}^{(\lambda)} \in \mathcal{M}_m(\mathbb{R})$ as follows.
For any $Z,Z' \in \mathbb{R}^m$, there holds
$$\langle \lambda, \sigma(Z,Z')\rangle = \langle  Z,\mathcal{U}^{(\lambda)} \cdot Z' \rangle.$$
If $(s_1,\dots,s_p)$ denotes an orthonormal basis of $\mathbb{R}^p$, we also define $\mathcal{U}_k \in \mathcal{M}_m(\mathbb{R})$ by
$$\langle s_k, \sigma(Z,Z')\rangle = \langle Z, \mathcal{U}_k \cdot Z' \rangle.$$
Conversely, the map $\sigma$ may be defined from $(\mathcal{U}_k)_{1 \leq k \leq p}$ thanks to the equality
$$\sigma(Z,Z') = \left( \langle Z, \mathcal{U}_k \cdot Z' \rangle \right)_{1\leq k \leq p}. $$
Notice that the map $\lambda \mapsto \mathcal{U}^{(\lambda)}$ is linear, with its image spanned by $(\mathcal{U}_k)_{1 \leq k \leq p}$.
As $\mathcal{U}^{(\lambda)}$ is an antisymmetric matrix, its rank is an even number.
We define the integer $d$ by
$$2d := \max_{\lambda \in \mathbb{R}^p} \text{ rank } \mathcal{U}^{(\lambda)}.$$
The set $\widetilde{\Lambda} := \{\lambda \in \mathbb{R}^p \: | \: \text{ rank } \mathcal{U}^{(\lambda)} = 2d\}$ is then a non empty Zariski-open subset of $\mathbb{R}^p$ - 
in particular, it is open and dense in $\mathbb{R}^p$ for the Euclidean topology.
Since the map $\lambda \mapsto \mathcal{U}^{(\lambda)}$ is continuous, there exist $d$ continuous functions
$$\eta_j : \mathbb{R}^p \to \mathbb{R}_+, \:  1 \leq j \leq d,$$ 
such that, in a suitable basis (see for instance \cite{Tyrtyshnikov}),
$\mathcal{U}^{(\lambda)}$ reduces to the form
$$
\left[
\begin{array}{ccc}
0 & \eta(\lambda) & 0 \\
-\eta(\lambda) & 0 & 0 \\
0 & 0 & 0
\end{array}
\right] \in \mathcal{M}_m(\mathbb{R}),
$$
where
$$\eta(\lambda) := \text{ diag }(\eta_1(\lambda),\dots,\eta_d(\lambda)) \in \mathcal{M}_d(\mathbb{R}).$$ 
We loosely denote by $(x_1(\lambda),\dots, x_d(\lambda), y_1(\lambda),\dots,y_d(\lambda),r_1(\lambda),\dots,r_t(\lambda))$ such a basis.
For readability purposes, we will often shorten the notation to $(x_1,\dots, x_d, y_1,\dots,y_d,r_1,\dots,r_t)$.

\subsection{A few examples}

A prime example of a $2$-step Lie group is given for $d \geq 1$ by the Heisenberg group $\mathbb{H}^d$, which is the set $\mathbb{R}^{2d} \times \mathbb{R}$ endowed with the group law
$$(Z,s)\cdot(Z',s') = (Z+Z',s+s'+\frac 12 \sigma_c(Z,Z')), $$
where $\sigma_c$ is the canonical symplectic form on $\mathbb{R}^d \times \mathbb{R}^d$.
For $x,y,x',y' \in \mathbb{R}^d$, 
$$\sigma_c((x,y),(x',y')) = \langle y,x'\rangle - \langle y',x\rangle, $$
where $\langle \cdot, \cdot \rangle$ denotes the usual scalar product on $\mathbb{R}^d$.
We present here another example.
Given the matrices 
$$
J = \left[
\begin{array}{cc}
0 & 1 \\
-1 & 0
\end{array}
\right],
S = \left[
\begin{array}{cc}
0 & 1 \\
1 & 0
\end{array}
\right] \in \mathcal{M}_2(\mathbb{R}),
$$
we define, for $\lambda = (\lambda_1,\lambda_2) \in \mathbb{R}^2$, the matrix
$$
\mathcal{U}^{(\lambda)} = \left[
\begin{array}{cc}
\lambda_1 J & \lambda_2 S \\
-\lambda_2 S & -\lambda_1 J
\end{array}
\right] \in \mathcal{M}_4(\mathbb{R}).
$$
On $\mathbb{R}^{4} \times \mathbb{R}^2$, we consider the group law generated by the matrices $\mathcal{U}^{(\lambda)}$.
That is, for $Z,Z' \in \mathbb{R}^4$ and $s = (s_1,s_2) ,s'=(s_1',s_2') \in \mathbb{R}^2$, we have
$$(Z,s)\cdot(Z',s') := \left(Z+Z',s_1+s_1'+ \frac 12 \left\langle Z, \left[
\begin{array}{cc}
 J & 0 \\
0 & - J
\end{array}
\right]Z' \right\rangle, s_2+s_2'+ \frac 12 \left\langle Z, \left[
\begin{array}{cc}
0 &  S \\
- S & 0
\end{array}
\right]Z' \right\rangle\right). $$
The positive eigenvalues of $\mathcal{U}^{(\lambda)}$ are $\eta_{\pm}(\lambda) = \left| |\lambda_1| \pm |\lambda_2| \right|$. 
In particular, $\eta_-(\lambda)$ vanishes on the straight lines $\{\lambda \in \mathbb{R}^2 \text{ s.t. }|\lambda_1| = |\lambda_2|\}$, whereas $\eta_+(\lambda) > 0$ for any $\lambda$ in $\mathbb{R}^2 \setminus \{0\}$.

\subsection{Definition of the Fourier transform}

We now turn to the practical aspects of the theory we aim at.
Given $(v_1,\dots,v_m)$ any basis of $\mathbb{R}^m$ and $(s_1,\dots,s_p)$ any basis of $\mathbb{R}^p$, a basis of $g$ is given by 
$\{V_i, 1 \leq i \leq m\} \cup \{S_k, 1 \leq k \leq p\},$ with
$$V_i := \partial_i - \frac 12 \sum_{k=1}^m (\mathcal{U}_k \cdot Z)_i \partial_{s_k} = \partial_i - \frac 12 \sum_{k,j=1}^m (\mathcal{U}_k)_{ij} Z_j \partial_{s_k},$$
$$S_k = \partial_{s_k}.$$ 
Choosing for $(v_1,\dots,v_m)$ a basis $(x_1,\dots,x_d,y_1,\dots,y_d,r_1,\dots,r_t)$, the family $(V_i)_{1 \leq i \leq m}$ decomposes as
$$X_j = \partial_{x_j} + \frac 12 \eta_j(\nabla_s)y_j \text{ for } 1 \leq j \leq d,$$
$$Y_j = \partial_{y_j} - \frac 12 \eta_j(\nabla_s)x_j \text{ for } 1 \leq j \leq d,$$
$$R_l = \partial_{r_l} \text{ for } 1 \leq l \leq t,$$
where the operator $\eta_j(\nabla_s)$ satisfies
$$\eta_j(\nabla_s)\left(e^{i\langle \lambda, \cdot \rangle} \right)(s) = i \eta_j(\lambda)e^{i\langle \lambda, s \rangle}.  $$
We define similarly the right-invariant vector fields $\widetilde{V}_i$ for $1 \leq i \leq m$ by
$$\widetilde{V}_i := \partial_i + \frac 12 \sum_{k=1}^m (\mathcal{U}_k \cdot Z)_i \partial_{s_k}.$$
In the basis $(x_1,\dots,x_d,y_1,\dots,y_d,r_1,\dots,r_t)$ defined above, the family $(\widetilde{V}_i)_{1 \leq i \leq m}$ decomposes as
$$\widetilde{X}_j = \partial_{x_j} - \frac 12 \eta_j(\nabla_s)y_j \text{ for } 1 \leq j \leq d,$$
$$\widetilde{Y}_j = \partial_{y_j} + \frac 12 \eta_j(\nabla_s)x_j \text{ for } 1 \leq j \leq d,$$
$$\widetilde{R}_l = R_l \text{ for } 1 \leq l \leq t.$$
For $(\lambda,\nu,w)$ in $\widetilde{\Lambda} \times \mathbb{R}^t \times \mathbb{R}^n$ with $w = (x,y,r,s)$, 
we define the irreducible unitary reprentations of $\mathbb{R}^n$ on $L^2(\mathbb{R}^d)$
$$\left(u_w^{\lambda,\nu}\phi\right)(\xi) := e^{-i \langle\nu, r\rangle} e^{-i\langle \lambda, s + [\xi + \frac x2,y]\rangle}\phi(\xi + x). $$
With these notations, the Fourier transform of the function $f \in L^1(\mathbb{R}^n)$ at the point 
$$(\lambda,\nu) \in \widetilde{\Lambda} \times \mathbb{R}^t$$ 
is a unitary operator acting on $L^2(\mathbb{R}^n)$ with
$$\mathcal{F}^g(f)(\lambda,\nu) := \int_{\mathbb{R}^n} f(w)u^{\lambda,\nu}_w dw. $$
Thinking of this operator as an infinite matrix, we look at its coefficients in a suitable basis.
For $n,m \in \mathbb{N}^d$ and $(\lambda,\nu) \in \widetilde{\Lambda} \times \mathbb{R}^t$, we let
$$\widetilde{\mathcal{F}_g}(f)(n,m,\lambda,\nu) := \left(\mathcal{F}^g(f)(\lambda,\nu)H_{m,\eta(\lambda)} \left| \right. H_{n,\eta(\lambda)} \right)_{L^2(\mathbb{R}^d)}. $$
Expanding out the scalar product, we notice the operator equality
$$\widetilde{\mathcal{F}_g}(f)(n,m,\lambda,\nu) = \mathcal{F}_{\mathbb{R}^t}\left(\mathcal{F}_g(f)(n,m,\lambda)\right)(\nu) =  
\mathcal{F}_g(\mathcal{F}_{\mathbb{R}^t}(f)(\nu))(n,m,\lambda),$$
where $\mathcal{F}_{\mathbb{R}^t}$ denotes the standard Fourier transform on the commutative group $(\mathbb{R}^t,+)$ and
$$\mathcal{F}_g(f)(n,m,\lambda) := 
\int_{(\mathbb{R}^d)^3 \times \mathbb{R}^p} 
f(x,y,s) e^{-i \langle \lambda, s\rangle} e^{-i \langle \eta(\lambda) \cdot (\xi + \frac x2), y\rangle}  
H_{m,\eta(\lambda)}(\xi + x) H_{n,\eta(\lambda)}(\xi) d\xi dx dy ds. $$
The action of the Fourier transform $\mathcal{F}_{\mathbb{R}^t}$ is already well-known and commutes with that of $\mathcal{F}_g$.
$$\text{Henceforth, we will assume that }t = 0 \text{ and will not mention }\nu \text{ anymore.}$$
We will focus solely on the properties of $\mathcal{F}_g$.
Performing an obvious change of variable inside the integral leads to the more symmetric form
$$\mathcal{F}_g(f)(n,m,\lambda) = 
\int_{\mathbb{R}^d \times \mathbb{R}^n} f(w) e^{-i \langle \lambda, s\rangle} e^{-i \langle \eta(\lambda) \cdot \xi, y\rangle}  
H_{m,\eta(\lambda)}\left(\xi + \frac x2\right) H_{n,\eta(\lambda)}\left(\xi - \frac x2\right) d\xi dw. $$
Denoting $\hat{w} = (n,m,\lambda)$ and letting
$$\mathcal{W}(\hat{w},x,y) := \int_{\mathbb{R}^d} e^{-i \langle \eta(\lambda) \cdot \xi, y\rangle}  
H_{m,\eta(\lambda)}\left(\xi + \frac x2\right) H_{n,\eta(\lambda)}\left(\xi - \frac x2\right) d\xi, $$
we see that
$$\mathcal{F}_g(f)(n,m,\lambda) = \int_{\mathbb{R}^n} \overline{e^{i \langle \lambda, s\rangle}\mathcal{W}(\hat{w},x,y)} f(w) dw.  $$
If one thinks of the family of functions $(\mathcal{W}(\hat{w},\cdot,\cdot))_{\hat{w} \in \mathbb{N}^{2d} \times \widetilde{\Lambda}}$ as a non commutative replacement of the characters on $\mathbb{R}^m$, then
the latter formula is very similar to that of the usual Fourier transform on $\mathbb{R}^n$.

\subsection{The frequency space}

Let us now describe what should be $\hat{g}$, the frequency space of $g$.
Since we have defined $\mathcal{F}_g(f)(n,m,\lambda)$ for $f \in L^1(\mathbb{R}^n)$, $n,m \in \mathbb{N}^d$ and $\lambda \in \widetilde{\Lambda}$, the set 
$$\widetilde{g}_E := \mathbb{N}^d \times \mathbb{N}^d \times \widetilde{\Lambda}$$
is a natural choice.
Endowing $\widetilde{g}_E$ with the distance
$$\rho_E((n,m,\lambda),(n',m',\lambda'))^2 := |\eta(\lambda)\cdot(n+m) - \eta(\lambda')\cdot(n'+m')|^2 + |(n-m) - (n'-m')|^2 + |\lambda-\lambda'|^2 $$
allows to account for the different types of decay (see Section \ref{sec:StudyFourierKernel} for more details).
Hence, the metric space $(\widetilde{g}_E,\rho_E)$ seems to be a reasonable candidate for the Fourier space of $g$.
However, it fails to be complete : for instance, denoting 
$$\Lambda_0 := \mathbb{R}^p \setminus \widetilde{\Lambda}, $$
points of the type $(0,0,\lambda_0) \in \mathbb{N}^{2d} \times \Lambda_0$ belong to the completion of $(\widetilde{g}_E,\rho_E)$. 
While it is possible to directly describe the completion of $(\widetilde{g}_E,\rho_E)$, writing exactly how we extend both $\widetilde{g}_E$ and $\rho_E$ is tedious.
It is comparatively easier to look first at \emph{Euclidean} isometric embedding of $(\widetilde{g}_E,\rho_E)$.
The set underlying the metric space, denoted by $\widetilde{g}$, will be an embedding of $\widetilde{g}_E$ into, say, $\mathbb{R}^n$.
The distance, however, will simply be a restriction of a standard Euclidean distance $|\cdot|$ on $\widetilde{g}$.
The main idea is that we reduce to a well-known distance, at the cost of a more intricate Fourier space.

We now make precise the ideas above.
For $\lambda \in \widetilde{\Lambda}$ and $1 \leq j \leq d$, define
$$g_j(\lambda) = \left( (\eta_j(\lambda) \cdot \mathbb{N}) \times \mathbb{Z}\right)_+ := 
\left\{(a_j,b_j) \in \mathbb{R}_+ \times \mathbb{Z} \:| \: \frac{a_j \pm \eta_j(\lambda) b_j}{2} \in \eta_j(\lambda)\cdot \mathbb{N}\right\} $$
and 
$$g(\lambda) := 
\left\{(a,b) \in (\mathbb{R}_+)^d \times \mathbb{Z}^d \:| \: \frac{a \pm \eta(\lambda)\cdot b}{2} \in \eta(\lambda)\cdot \mathbb{N}^d\right\}= \prod_{j=1}^d g_j(\lambda). $$
A way to think about the set $g_j(\lambda)$ is the following.
For any $b_j \in \mathbb{Z}$ and $\lambda \in \widetilde{\Lambda}$, we have 
$$ \{a_j \in \mathbb{R}_+ \: | \: (a_j,b_j) \in g_j(\lambda)\} = \{(2n+|b_j|)\eta_j(\lambda) \: | \: n \in \mathbb{N}\}.$$
Otherwise said, given $b_j$ and $\lambda$, the admissible $a_j$'s form a half-infinite comb of width $2\eta_j(\lambda)$ starting at $\eta_j(\lambda)|b_j|$.
The set $\widetilde{g}$ is now defined by
$$\widetilde{g} := \bigsqcup_{\lambda \in \widetilde{\Lambda}} g(\lambda) \times \{\lambda\}. $$
As previously explained, we endow $\widetilde{g}$ with the distance inherited from the Euclidean distance on $\mathbb{R}^d \times \mathbb{Z}^d \times \mathbb{R}^p$.
The isometry between $\widetilde{g}_E$ and $\widetilde{g}$ is given by
$$i_E :
\left\{
\begin{array}{ccc}
 \widetilde{g}_E &\longrightarrow& \widetilde{g}\\
 (n,m,\lambda) &\longmapsto& (\eta(\lambda)\cdot (n+m), n-m, \lambda).
\end{array}
\right.
$$
Of course, as $\widetilde{g}_E$ was not complete, $\widetilde{g}$ is not either.
We finally describe the completion of $(\widetilde{g},|\cdot|)$.
Consider a sequence $(\lambda_p)_p \subset \widetilde{\Lambda}$ converging to $\lambda_0 \in \Lambda_0$ for which $\eta_j(\lambda_0) = 0$.
At least formally, we expect the constraint defining $g_j(\cdot)$ to become vacuous and hence, it seems natural to let
$$g_j(\lambda_0) = ((0 \cdot \mathbb{N}) \times \mathbb{Z})_+ := \mathbb{R}_+ \times \mathbb{Z}. $$
We then define as before, for $\lambda \in \mathbb{R}^p$,
$$g(\lambda) := \prod_{j=1}^d g_j(\lambda).$$
With this convention, we expect the completion of $\widetilde{g}$ for $|\cdot|$ to be
$$\hat{g} :=  \bigsqcup_{\lambda \in \mathbb{R}^p} g(\lambda) \times \{\lambda\}.$$
This is indeed true, as shown by Proposition \ref{PropositionCompletiongChapeau}.

\section{Description of the results}

The main goal of this paper is to establish a familiar Fourier theory on stratified nilpotent Lie groups of step $2$.
We begin from the well-known representation-theoretic Fourier transform, whose main properties are recalled in Appendix \ref{appendix:FourierRepresentation}.
In particular, this Fourier transofmr is an intertwining operator for the subelliptic laplacian, acting on functions on $g$ 
and a rescaled version of the quantum harmonic oscillator, acting on functions on $\mathbb{R}^d$.
Since rescaled Hermite functions form an eigenbasis of $L^2(\mathbb{R}^d)$ for the quantum harmonic oscillator, 
it is natural to expand the representation-theoretic Fourier transform on a rescaled Hermite basis.
This idea, borrowed from \cite{BahouriCheminDanchinHeisenberg}, leads to the definition of a frequency space $\widetilde{g}$.
This frequency space being noncomplete, we describe its completion $\hat{g}$ and prove some of its properties in Section \ref{sec:TopologyMeasure}.

In Section \ref{sec:StudyFourierKernel}, we prove that the so-called 'Fourier kernel' $\mathcal{W}$ does possess a continuous extension from $\widetilde{g}$ to $\hat{g}$.
Moreover, we give an explicit expression of the kernel at the boundary points of $\hat{g}$ as a power series involving some combinatorial quantities $F_{\ell_1,\ell_2}(b_j)$.
These quantities are known in the literature as the alternate Vandermonde convolution (see [reference ?] for more detials).

In Section \ref{sec:IndependantCentralVariable}, we prove a lemma on functions approximating a Dirac mass in their last variable.
As a consequence, we are able to define the Fourier transform of functions independant of the central variable, 
much as we may define the Fourier transform of functions on $\mathbb{R}^n$ independant of one variable.

In Section \ref{sec:ComputingKernelBoundary}, we give an integral representation formula for the Fourier kernel as the boundary, which may be of independant interest.
Finding this formula relies on several space-frequency properties of the original Fourier kernel $\mathcal{W}$ 
akin to the familiar derivation-multiplication duality in the torus or the whole space for the usual, commutative Fourier transform.

The interested reader will find in Appendix $\ref{appendix:Hermite}$ some standard computations and definitions involving the Hermite functions.

\section{Topology and measure theory on $\hat{g}$}
\label{sec:TopologyMeasure}
\subsection{The completion of the frequency space.}
We begin by giving some foundation to the theory, by proving that the completion of the natural frequency space $\widetilde{g}$ 
for the distance $|\cdot|$ is indeed what it ought to be.
\begin{proposition}
 The closure of $\widetilde{g}$ in $(\mathbb{R}_+)^d \times \mathbb{Z}^d \times \mathbb{R}^p$ for the distance $|\cdot|$ is equal to $\hat{g}$.
 \label{PropositionCompletiongChapeau}
\end{proposition}
\begin{proof}
 Let $(a(q),b(q),\lambda(q))_{q \in \mathbb{N}}$ be a Cauchy sequence in $\widetilde{g}$.
 Since $\widetilde{g}$ is a subset of $(\mathbb{R}_+)^d \times \mathbb{Z}^d \times \mathbb{R}^p$, which is complete for the distance $|\cdot|$, there exists 
 $(a,b,\lambda) \in (\mathbb{R}_+)^d \times \mathbb{Z}^d \times \mathbb{R}^p$ such that
 $$(a(q),b(q),\lambda(q)) \longrightarrow (a,b,\lambda) \in (\mathbb{R}_+)^d \times \mathbb{Z}^d \times \mathbb{R}^p \text{ as } q \to \infty. $$
 To simplify the exposition of the proof, we look separately at each component $g_j(\lambda)$ for $1 \leq j \leq d$ and distinguish between two cases.
 \begin{itemize}
  \item Assume that $\eta_j(\lambda) \neq 0$. 
  Then, there exists $c > 0$ such that for all $q \in \mathbb{N}$,
  $$\eta_j(\lambda(q)) \geq c. $$
  Hence, the two sequences of integers 
  $$\left(\frac{\eta_j(\lambda(q))^{-1}a_j(q) \pm b_j(q)}{2}\right)_{q\in \mathbb{N}}$$
  are also Cauchy sequences.
  Passing to the limit in the equations above give
  $$\frac{\eta_j(\lambda)^{-1}a_j \pm b_j}{2} \in \mathbb{N},$$
  which exactly says that $(a_j,b_j) \in g_j(\lambda)$.
  \item Assume now that $\eta_j(\lambda) = 0$.
  Since $(b_j(q))_{q\in\mathbb{N}}$ is a Cauchy sequence of integers, we immediately get $b \in \mathbb{Z}$.
  On the other hand, as $(a(q))_{q\in\mathbb{N}}$ is a Cauchy sequence in $\mathbb{R}_+$, we have $a\in \mathbb{R}_+$.
  Thus, we again have $(a,b) \in g_j(\lambda)$, this time in the extended sense. 
 \end{itemize}
 Up to now, we have shown that $\hat{g}$ contains the closure of $\widetilde{g}$ for the Euclidean distance.
 Conversely, let $(a,b,\lambda_0) \in \hat{g} \setminus \widetilde{g}$ (in particular, $\lambda_0 \in \Lambda_0$).
 Denote 
 $$J(\lambda_0) := \{1 \leq j \leq d \: | \: \eta_j(\lambda_0) = 0\}. $$
 Since $\widetilde{\Lambda}$ is dense in $\mathbb{R}^p$, there exists a sequence $(\lambda(q))_{q\in\mathbb{N}} \subset \widetilde{\Lambda}$ with
 $$\lambda(q) \to \lambda_0 \text{ as } q \to \infty.$$
 Regarding $b$, we simply let $b(q) = b$ for all $q \in \mathbb{N}$.
 We again distinguish between two cases to define the sequence $(a_j(q))_{q\in\mathbb{N}}$.
 \begin{itemize}
  \item If $j \not\in J(\lambda_0)$, there exists $h^-_j \in \mathbb{N}$ such that 
  $$ \frac{a_j - \eta_j(\lambda_0)b_j}{2} = h^-_j\eta_j(\lambda_0).$$
  Defining for $q \in \mathbb{N}$
  $$a_j(q) := \eta_j(\lambda(q))(2h^-_j+b_j),$$
  we see that, for any $q\in\mathbb{N}$,
  $$\frac{a_j(q) - \eta_j(\lambda(q))b_j}{2} \in \eta_j(\lambda(q))\cdot \mathbb{N}.$$
  Hence, for any $q\in \mathbb{N}$, we have $(a_j(q),b_j,\lambda(q)) \in g_j(\lambda(q))$ and moreover, as $q \to \infty$,
  $$a_j(q) \to  \eta_j(\lambda_0)(2h^-+b) = a_j.$$
  \item If $j \in J(\lambda_0)$, we use a similar strategy.
  Let $(h^-_j(q))_{q\in\mathbb{N}}$ be a sequence of integers tending to infinity such that $2h^-_j(q)\eta_j(\lambda(q)) \to a_j$ as $q \to \infty$. 
  Similarly to the first case, we define, for $q \in \mathbb{N}$,
  $$a_j(q) := \eta_j(\lambda(q))(2h^-_j(q)+b_j). $$
  In the particular case where $a_j = 0$, we do not want to let $h^-_j(q) \equiv 0$ for all $q \in \mathbb{N}$, so as to ensure that $2h^-_j(q) + b_j \geq 0$ for $q$ big enough.
  Of course, we have
  $$(a_j(q),b_j)\in g_j(\lambda(q)) $$
  for $q$ big enough and
  $$a_j(q) \to a_j \text{ as } q \to \infty. $$
 \end{itemize}
Gathering what we did for each coordinate, we have found a sequence $(a(q),b,\lambda(q))_{q\in\mathbb{N}} \subset \widetilde{g}$ converging to $(a,b,\lambda)$ for $|\cdot|$.
This closes the proof.
\end{proof}
As a consequence of Proposition \ref{PropositionCompletiongChapeau}, $\hat{g}$ is a closed subset of $(\mathbb{R}_+)^d \times \mathbb{Z}^d \times \mathbb{R}^p$ for the standard Euclidean distance.
Hence, it is trivially locally compact.

\subsection{A measure on $\hat{g}$}
Owing to the fibered-looking structure of $\hat{g}$, it seems reasonable to define a measure on it through its disintegration on each $g(\lambda)$ for $\lambda \in \mathbb{R}^p$.
 In turn, defining a measure on each $g_j(\lambda)$ for $1 \leq j \leq d$ immediately gives rise to a measure on $g(\lambda)$, simply by taking the tensor product.
 Finally, denoting 
 $$g_{j,b_j}(\lambda) := \{a_j \in \mathbb{R}_+ \: | \: (a_j,b_j) \in g_j(\lambda)\}, $$
 we have the decomposition 
 $$g_j(\lambda) := \bigsqcup_{b_j \in \mathbb{Z}} g_{j,b_j}(\lambda) \times \{b_j \}.$$ 
 We now construct a measure $d\hat{w}$ on $\hat{g}$ following a bottom-up procedure, starting from the $g_{j,b_j}(\lambda)$ and ending with $\hat{g}$.
 Let $1 \leq j \leq d$ and $\lambda \in \mathbb{R}^p$ be such that $\eta_j(\lambda) \neq 0$.
 Given $b_j \in \mathbb{Z}$ and
 $$\theta \in \mathcal{C}_c(g_{j,b_j}(\lambda)),$$
 the measure $d\mu^{\lambda}_{j,b_j}$ on $g_{j,b_j}(\lambda)$ is defined by the equality
 $$\int_{g_{j,b_j}(\lambda)} \theta(a_j) d\mu^{\lambda}_{j,b_j}(a_j) := 2\eta_j(\lambda) \sum_{a_j \in g_{j,b_j}(\lambda)} \theta(a_j). $$
 That is, in this case, $d\mu^{\lambda}_{j,b_j}$ is simply a rescaled version of the counting measure on the discrete set $g_{j,b_j}(\lambda)$.
 If $\eta_j(\lambda) = 0$, we simply let $d\mu^{\lambda}_{j,b_j}$ be the Lebesgue measure on $g_{j,b_j}(\lambda)$, which is none other that $\mathbb{R}_+$.
 Now, the measure $d\mu^{\lambda}_j$ is defined as the integration of the family $(d\mu^{\lambda}_{j,b_j})_{b_j\in \mathbb{Z}}$ 
 with respect to the counting measure on $\mathbb{Z}$.
 That is, given $\theta \in \mathcal{C}_c(g_j(\lambda))$, we have
 $$\int_{g_j(\lambda)} \theta(a_j,b_j) d\mu^{\lambda}_j(a_j,b_j) := \sum_{b_j \in \mathbb{Z}} \int_{g_{j,b_j}(\lambda)} \theta(a_j,b_j) d\mu^{\lambda}_{j,b_j}(a_j).  $$
 On $g(\lambda)$, we define the measure $d\mu^{\lambda}$ as the tensor product of the $d\mu^{\lambda}_j$, that is
 $$d\mu^{\lambda} := \bigotimes_{j=1}^d d\mu^{\lambda}_j.$$
 Then, the measure $d\hat{w}$ on $\hat{g}$ is the integration of the family $(d\mu^{\lambda})_{\lambda \in \mathbb{R}^p}$ with respect to the Lebesgue measure on $\mathbb{R}^p$.
 That is, for $\theta \in \mathcal{C}_c(\hat{g})$, we have
 $$\int_{\hat{g}} \theta(\hat{w}) d\hat{w} := \int_{\mathbb{R}^p} \left( \int_{g(\lambda)}\theta(a,b,\lambda) d\mu^{\lambda}(a,b)\right) d\lambda. $$
 The relevance of these successive definitions is summarized by the following proposition.
 \begin{proposition}
  The measure field $\lambda \mapsto d\mu^{\lambda}$ is weak-$*$ continuous on $\mathbb{R}^p$.
 \end{proposition}
 The proof of this proposition is immediate once we have the next lemma at hand; for this reason, we will only prove the lemma.
 \begin{lemma}
 Let $\lambda_0 \in \Lambda_0$ with, for instance, $\eta_j(\lambda_0) = 0$.
 Let $b_j \in \mathbb{Z}$.
  Then, if $\lambda \to \lambda_0$, we have
  $$d\mu^{\lambda}_{j,b_j} \rightharpoonup^* d\mu^{\lambda_0}_{j,b_j}$$
  in the weak sense of measures.
 \end{lemma}
\begin{proof}
 Let $\theta \in \mathcal{C}_c(\mathbb{R}_+)$.
 By definition of $d\mu^{\lambda}_{j,b_j}$, we have
 $$\int_{\mathbb{R}_+} \theta(a_j) d\mu^{\lambda}_{j,b_j}(a_j) = 2\eta_j(\lambda) \sum_{n \in \mathbb{N}}\theta(\eta_j(\lambda)(2n+|b_j|)).$$
 Thanks to the continuity of $\theta$ and the obvious fact that $\eta_j(\lambda)|b_j| \to 0$ as $\lambda \to \lambda_0$, we have
  $$\int_{\mathbb{R}_+} \theta(a_j) d\mu^{\lambda}_{j,b_j}(a_j) = 2\eta_j(\lambda) \sum_{n \in \mathbb{N}}\theta(2\eta_j(\lambda)n) + o(1).$$
Since $\theta$ is continuous and compactly supported, the above sum is nothing else than a Riemann sum.
Hence, as $\lambda \to \lambda_0$,
$$2\eta_j(\lambda) \sum_{n \in \mathbb{N}}\theta(2\eta_j(\lambda)n) \longrightarrow \int_{\mathbb{R}_+} \theta(a_j)da_j = \langle d\mu^{\lambda_0}_{j,b_j},\theta\rangle. $$
\end{proof}

 \section{A study of the Fourier kernel.}
\label{sec:StudyFourierKernel}
Let 
$$\Theta : (\hat{w},w) \mapsto e^{i\langle \lambda,s\rangle} \mathcal{W}(\hat{w},x,y). $$
In this section we study closely the properties of the Fourier kernel $\Theta$. 
We begin by proving some identities linking its regularity in the spatial variables with its decay in the Fourier variables.
These identities are the justification of the 'regularity implies decay' motto, common in (commutative) Fourier analysis.
Since the computations performed in this section rely on properties of the (rescaled) Hermite functions, we temporarily parametrize $\widetilde{g}$ by $\widetilde{g}_E$.
Explicitly, we let, for $(a,b,\lambda) \in \widetilde{g}$,
$$n := \frac{\eta(\lambda)\cdot a - b}{2} $$
and
$$m := \frac{\eta(\lambda)\cdot a + b}{2}. $$
\subsection{Regularity and decay of $\Theta$}
Applying the vector fields $X_j$ and $Y_j$ to $\Theta$, we get
$$X_j(\Theta)(\hat{w},w) = e^{i\langle \lambda,s\rangle} \left(\partial_{x_j}\mathcal{W} + \frac 12 i \eta_j(\lambda) y_j\mathcal{W}\right)(\hat{w},x,y) $$
and
$$Y_j(\Theta)(\hat{w},w) = e^{i\langle \lambda,s\rangle} \left(\partial_{y_j}\mathcal{W} - \frac 12 i \eta_j(\lambda) x_j\mathcal{W}\right)(\hat{w},x,y). $$
After some computations, we find that 
 $$\left(\partial_{x_j}\mathcal{W} + \frac 12 i \eta_j(\lambda) y_j\mathcal{W}\right)(\hat{w},x,y) \\
= - \int_{\mathbb{R}^d} e^{i \langle \eta(\lambda) \cdot \xi, y\rangle} H_{m,\eta(\lambda)}
\left( \xi + \frac x2\right)  \partial_{\xi_j}\left( H_{n,\eta(\lambda)}\left(\xi - \frac x2\right) \right) d\xi.$$
Similarly,
\begin{multline*}
  \left(\partial_{y_j}\mathcal{W} - \frac 12 i \eta_j(\lambda) x_j\mathcal{W}\right)(\hat{w},x,y) \\
= i \eta_j(\lambda) \int_{\mathbb{R}^d} e^{i \langle \eta(\lambda) \cdot \xi, y\rangle} H_{m,\eta(\lambda)}
\left( \xi + \frac x2\right)  \left(\xi_j - \frac 12 x_j\right)\left( H_{n,\eta(\lambda)}\left(\xi - \frac x2\right) \right) d\xi. 
\end{multline*}
In particular, owing to Equation $\eqref{RescaledHarmonicOscillator}$, we have
\begin{equation}
 (X_j^2 + Y_j^2)(\Theta)(\hat{w},w) = (-2n_j + 1) \eta_j(\lambda)\Theta(\hat{w},w).
 \label{DecayLaplacien}
\end{equation}
Arguing similarly for the right-invariant vector fields, we readily get 
$$(\widetilde{X}_j^2 + \widetilde{Y}_j^2)(\Theta)(\hat{w},w) = (-2m_j + 1) \eta_j(\lambda)\Theta(\hat{w},w).$$
Subtracting the two lines gives
$$(X_j^2 + Y_j^2 - \widetilde{X}_j^2 - \widetilde{Y}_j^2)(\Theta)(\hat{w},w) = 2 \eta_j(\lambda) (m_j - n_j)\Theta(\hat{w},w).$$
On the other hand, direct computations give 
$$X_j^2 - \widetilde{X}_j^2 = - 2 \sum_k (\mathcal{U}_k \cdot Z)_j \partial_{x_j} \partial_{s_k}, $$
whence
$$(X_j^2 - \widetilde{X}_j^2)(\Theta)(\hat{w},w)
= - 2 i (\mathcal{U}^{(\lambda)} \cdot Z)_j \partial_{x_j} \Theta (\hat{w},w)
= -2 i \eta_j(\lambda) y_j \partial_{x_j} \Theta (\hat{w},w).$$
Similarly,
$$(Y_j^2 - \widetilde{Y}_j^2)(\Theta)(\hat{w},w)
= 2 i \eta_j(\lambda) x_j \partial_{y_j} \Theta(\hat{w},w).$$
Denoting $\mathcal{T}_j := x_j \partial_{y_j} - y_j \partial_{x_j}$, we have shown the equality
$$2 \eta_j(\lambda) (m_j - n_j)\Theta(\hat{w},w)
= 2i \eta_j(\lambda) \mathcal{T}_j(\Theta)(\hat{w},w),$$
which becomes
\begin{equation}
 (m_j - n_j) \mathcal{W}(\hat{w},w)
= i \mathcal{T}_j(\mathcal{W})(\hat{w},w).
\label{DecayEnk}
\end{equation}
Finally, it is clear that
\begin{equation}
 \nabla_s (\Theta)(\hat{w},w) = i\lambda \Theta(\hat{w},w).
 \label{DecayLambda}
\end{equation}
Equations $\eqref{DecayLaplacien}, \eqref{DecayEnk}$ and $\eqref{DecayLambda}$ justify the choice of the distance $\rho_E$ on $\widetilde{g}_E$.
Together, they account for all decay aspects of $\Theta$ in the variable $\hat{w}$.

\subsection{Continuous extension of $\mathcal{W}$ to $\hat{g}$}

To study the continuity of $\mathcal{W}$, first write 
$$\mathcal{W}(\hat{w},x,y) = e^{\frac i2 \langle \eta(\lambda)\cdot  x, y\rangle} \widetilde{\mathcal{W}}(\hat{w},x,y).$$
The new function $\widetilde{\mathcal{W}}$ is a tensor product, as $\mathcal{W}$ was.
Denoting 
$$\widetilde{\mathcal{W}}_j(a_j,b_j,\lambda,x_j,y_j) := \int_{\mathbb{R}} e^{i \eta_j(\lambda) \xi_j y_j} H_{m_j,\eta_j(\lambda)}(\xi_j + x_j)H_{n_j,\eta_j(\lambda)}(\xi_j) d\xi_j.$$ 
we only have to exhibit a continuous extension of $\widetilde{\mathcal{W}}_j$ to (bounded sets of) $\hat{g}$.
Let us begin with a series expansions for $\widetilde{\mathcal{W}}_j$.
\begin{proposition}
 For any $\lambda \in \mathbb{R}^p$, $(a_j,b_j) \in g_j(\lambda)$ and $x_j,y_j \in \mathbb{R}$, we have
 $$\widetilde{\mathcal{W}}_j(a_j,b_j,\lambda,x_j,y_j) = \sum_{\ell_1,\ell_2 = 0}^{\infty} \eta_j(\lambda)^{\frac{\ell_1 + \ell_2}{2}} \frac{(iy_j)^{\ell_1}x_j^{\ell_2}}{\ell_1!\ell_2!}
\left(M^{\ell_1}_j H_{m_j} \left| \right.  H_{n_j}^{(\ell_2)} \right)_{L^2(\mathbb{R})}.$$
\label{SeriesExpansionW}
\end{proposition}
Since we are interested in bounded subsets of $\widetilde{g}$, let us define, for $r > 0$,
$$\mathcal{B}_j(r) := \{(a_j,b_j,\lambda,x_j,y_j) \in g_j(\lambda)\times\{\lambda\} \times \mathbb{R}^2 \text{ s.t. } |a_j|+|b_j|^2+|\lambda|^2+|x_j|^2+|y_j|^2 \leq r^2 \}.$$
We will, in addition, need to bound locally the function $\eta_j$.
Define
$$C_{\eta_j} := \sup_{|\lambda| = 1} \eta_j(\lambda),$$
which is finite thanks to the continuity of $\eta_j$.
The homogeneity of $\eta_j$ entails for all $\lambda \in \mathbb{R}^p$ the bound
$$\eta_j(\lambda) \leq C_{\eta_j}|\lambda|. $$
The proof of Proposition \ref{SeriesExpansionW} requires the following lemma.
\begin{lemma}
 For any $r > 0$, the function
 $$(a_j,b_j,\lambda,x_j,y_j) \mapsto \sum_{\ell_1,\ell_2 = 0}^{\infty} \frac{\eta_j(\lambda)^{\frac{\ell_1+\ell_2}{2}} |y_j|^{\ell_1}|x_j|^{\ell_2}}{\ell_1!\ell_2!} 
 \|H_{m_j}^{(\ell_2)}\|_{L^2(\mathbb{R})} \|M^{\ell_1}_jH_{n_j}\|_{L^2(\mathbb{R})}$$ 
 is well-defined and the underlying series converges normally, as a function of $(a_j,b_j,\lambda,x_j,y_j)$, on $\mathcal{B}_j(r)$.
 \label{NormalConvergenceW}
\end{lemma}
\begin{proof}
Arguing by induction (see the details in the Appendix, Lemma \ref{RecurrenceChianteHermite}), we get the bounds 
$$\|M^{\ell_1}_jH_{n_j}\|_{L^2(\mathbb{R})} \leq (2n_j + 2\ell_1)^{\frac{\ell_1}{2}} $$
and
$$\|H_{m_j}^{(\ell_2)}\|_{L^2(\mathbb{R})} \leq (2m_j + 2\ell_2)^{\frac{\ell_2}{2}}. $$
Thus, 
$$\frac{\eta_j(\lambda)^{\frac{\ell_1}{2}} |y_j|^{\ell_1}}{\ell_1!} \|M^{\ell_1}_jH_{n_j}\|_{L^2(\mathbb{R})}
\leq \frac{ r^{\ell_1}}{\ell_1!} (2r + 2C_{\eta_j}r\ell_1)^{\frac{\ell_1}{2}}
\leq \frac{(2C_{\eta_j} r^{\frac 32})^{\ell_1}}{\ell_1!} (1 + \ell_1)^{\frac{\ell_1}{2}}.$$
Similarly, 
$$\frac{\eta_j(\lambda)^{\frac{\ell_2}{2}} |x_j|^{\ell_2}}{\ell_2!} \|H_{m_j}^{(\ell_2)}\|_{L^2(\mathbb{R})}
\leq \frac{ r^{\ell_2}}{\ell_2!} (2r + 2C_{\eta_j}r\ell_2)^{\frac{\ell_2}{2}}
\leq \frac{(2C_{\eta_j} r^{\frac 32})^{\ell_2}}{\ell_2!} (1 + \ell_2)^{\frac{\ell_2}{2}}.$$
Owing to the Stirling equivalent, we have
$$\sum_{\ell_1,\ell_2 = 0}^{\infty}  \frac{(2C_{\eta_j} r^{\frac 32})^{\ell_1+\ell_2}}{\ell_1! \ell_2!} (1 + \ell_1)^{\frac{\ell_1}{2}}(1 + \ell_2)^{\frac{\ell_2}{2}} < \infty,$$ 
thereby proving the desired convergence.
\end{proof}
We now return to the proof of Proposition \ref{SeriesExpansionW}.
\begin{proof}[Proof of Proposition \ref{SeriesExpansionW}]
We begin by expanding the exponential in its power series.
We have
$$\widetilde{\mathcal{W}}_j(a_j,b_j,\lambda,x_j,y_j) = 
\int_{\mathbb{R}} \sum_{\ell_1 = 0}^{\infty} \frac{(i \eta_j(\lambda) \xi_j y_j)^{\ell_1}}{\ell_1!} H_{m_j,\eta_j(\lambda)}(\xi_j + x_j)H_{n_j,\eta_j(\lambda)}(\xi_j) d\xi_j.$$ 
Since $H_{m_j}$ is an entire function, we may expand it as 
$$H_{m_j,\eta_j(\lambda)}(\xi_j + x_j) = \sum_{\ell_2 = 0}^{\infty} \frac{(\eta_j(\lambda)^{\frac 12}x_j)^{\ell_2}}{\ell_2!} \left(H_{m_j}^{(\ell_2)}\right)_{\eta_j(\lambda)}(\xi_j) $$
and get
$$\widetilde{\mathcal{W}}_j(a_j,b_j,\lambda,x_j,y_j) 
= \int_{\mathbb{R}}  \sum_{\ell_1 = 0}^{\infty}  \sum_{\ell_2 = 0}^{\infty} \frac{\eta_j(\lambda)^{\frac{\ell_1+\ell_2}{2}} (iy_j)^{\ell_1}x_j^{\ell_2}}{\ell_1!\ell_2!} 
 \left(H_{m_j}^{(\ell_2)}\right)_{\eta_j(\lambda)}(\xi_j)(M^{\ell_1}_jH_{n_j,\eta_j(\lambda)})(\xi_j) d\xi_j. $$
Rescaling the integration variable, we also have
$$\widetilde{\mathcal{W}}_j(a_j,b_j,\lambda,x_j,y_j) 
= \int_{\mathbb{R}}  \sum_{\ell_1 = 0}^{\infty}  \sum_{\ell_2 = 0}^{\infty} \frac{\eta_j(\lambda)^{\frac{\ell_1+\ell_2}{2}} (iy_j)^{\ell_1}x_j^{\ell_2}}{\ell_1!\ell_2!} 
H_{m_j}^{(\ell_2)}(\xi_j)(M^{\ell_1}_jH_{n_j})(\xi_j) d\xi_j. $$
Thanks to the lemma, Fubini's theorem applies and allows us to exchange the integrals with the sums.
The series expansion of $\widetilde{\mathcal{W}}_j$ is now justified.
\end{proof}
For $(\ell_1,\ell_2) \in \mathbb{N}^2$ and $1 \leq j \leq d$, let 
$$\mathcal{H}_{\ell_1,\ell_2,j} := 
\left\{
\begin{array}{rrcll}
  \bigsqcup_{\lambda \in \widetilde{\Lambda}} g_j(\lambda) \times \{\lambda\}& &\longrightarrow& &\mathbb{R}\\
 (a_j,b_j,\lambda)& &\longmapsto& &\eta_j(\lambda)^{\frac{\ell_1+\ell_2}{2}}\left(M^{\ell_1} H_{m_j} \left| \right.  H_{n_j}^{(\ell_2)} \right)_{L^2(\mathbb{R})}.
\end{array}
\right.
$$
With the above series expansion for $\widetilde{\mathcal{W}}_j$ at hand, we study the $(\mathcal{H}_{\ell_1,\ell_2,j})_{\ell_1,\ell_2 \in \mathbb{N}}$.
Defining for $k _in \mathbb{Z}$
$$F_{\ell_1,\ell_2}(k) := \sum_{\ell_1'= 0}^{\ell_1}\sum_{\ell'_2 = 0}^{\ell_2} (-1)^{\ell_2 - \ell_2'} \binom{\ell_1}{\ell_1'} \binom{\ell_2}{\ell_2'}
\mathds{1}_{\{2(\ell_1' + \ell_2') = k + \ell_1 + \ell_2\}},$$
we prove the following.
\begin{proposition}
 For any $\ell_1,\ell_2 \in \mathbb{N}$, the function $\mathcal{H}_{\ell_1,\ell_2,j}$ is continuous on $\bigsqcup_{\lambda \in \widetilde{\Lambda}} g_j(\lambda) \times \{\lambda\}$.
 Moreover, given $\lambda_0 \in \eta_j^{-1}(\{0\})$, if 
 $$(a_j,b_j,\lambda) \to (a_j,b_j,\lambda_0) \in \mathbb{R}_+ \times \mathbb{Z} \times \eta_j^{-1}(\{0\}),$$
 then
 $$\mathcal{H}_{\ell_1,\ell_2,j}(a_j,b_j,\lambda) \longrightarrow  \left(\frac{a_j}{4} \right)^{\frac{\ell_1 + \ell_2}{2}} F_{\ell_1,\ell_2}(b_j).$$
 \label{ContinuiteH}
\end{proposition}
In order to handle points close to the zero set of $\eta_j$, we will need the following lemma. 
We refer the reader to Appendix \ref{appendix:Hermite} for the definition of the creation and annihilation operators $A_j$ and $C_j$.
\begin{lemma}
Given $\ell \in \mathbb{N}$, assume that $\eta_j(\lambda) \to 0$ and $\eta_j(\lambda)n_j \to \frac{a_j}{2} \in \mathbb{R}_+$. Then
 $$\left\|\eta_j(\lambda)^{\frac{\ell}{2}} \left(\frac{A_j \pm C_j}{2}\right)^{\ell}H_{n_j} -
 \left(\frac{a_j}{4}\right)^{\frac{\ell}{2}} \sum_{\ell' = 0}^{\ell} (\pm1)^{\ell'}\binom{\ell}{\ell'}H_{n_j + 2\ell'- \ell}\right\|_{L^2(\mathbb{R})}
 \longrightarrow 0.$$
 \label{LemmeCommutationAsymptotique}
\end{lemma}
\begin{proof}
 Let us define 
 $$\mathcal{R}_{\ell,n_j,1}^{\pm} := \eta_j(\lambda)^{\frac{\ell}{2}} (A_j \pm C_j)^{\ell}H_{n_j}, $$
 $$\mathcal{R}_{\ell,n_j,2}^{\pm} :=  a_j^{\frac{\ell}{2}} \sum_{\ell' = 0}^{\ell} (\pm1)^{\ell'}\binom{\ell}{\ell'}H_{n_j + 2\ell'- \ell}. $$
 We obviously have $\mathcal{R}_{0,n_j,1}^{\pm} = \mathcal{R}_{0,n_j,2}^{\pm} = H_{n_j}$ and thus, $\|\mathcal{R}_{0,n_j,1}^{\pm} - \mathcal{R}_{0,n_j,2}^{\pm}\|_{L^2(\mathbb{R})} = 0$.
 By definition, we also have 
 $$\sqrt{\eta_j(\lambda)} (A\pm C)\mathcal{R}_{\ell,n_j,1}^{\pm} = \mathcal{R}_{\ell+1,n_j,1}^{\pm}.$$
 It remains to study the effect of $\sqrt{\eta_j(\lambda)} (A_j\pm C_j)$ on $\mathcal{R}_{\ell,n_j,2}^{\pm}$ to conclude.
 Recalling Equations $\eqref{CreationHermite}$ and $\eqref{AnnihilationHermite}$ in the Appendix along with our assumptions on $\eta_j(\lambda)$ and $n_j$, we have 
 $$\sqrt{\eta_j(\lambda)} A_jH_{n_j + 2\ell' - \ell} = 
 \sqrt{\eta_j(\lambda)} \left(\sqrt{2(n_j+2\ell'-\ell)}H_{n_j + 2\ell'-\ell-1} \right) = (\sqrt{a_j}+o(1))H_{n_j + 2\ell'-\ell-1} $$
 and
  $$\sqrt{\eta_j(\lambda)} C_jH_{n_j + 2\ell' - \ell} = 
 \sqrt{\eta_j(\lambda)} \left(\sqrt{2(n_j + 2\ell'-\ell+1)}H_{n_j+2\ell'-\ell+1} \right) = (\sqrt{a_j}+o(1))H_{n_j + 2\ell'-\ell+1}.$$
 In particular, we have
 \begin{multline}
  \left\|\sqrt{\eta_j(\lambda)} (A_j\pm C_j) (\mathcal{R}_{\ell,n_j,1}^{\pm}-\mathcal{R}_{\ell,n_j,2}^{\pm}) \right\|_{L^2(\mathbb{R})} \\
  \leq (\sqrt{a_j}+o(1)) \left(\|\mathcal{R}_{\ell,n_j-1,1}^{\pm} -\mathcal{R}_{\ell,n_j-1,2}^{\pm}\|_{L^2(\mathbb{R})}
  +\|\mathcal{R}_{\ell,n_j+1,1}^{\pm} -\mathcal{R}_{\ell,n_j+1,2}^{\pm}\|_{L^2(\mathbb{R})}\right).
  \label{BornitudeReste}
 \end{multline}
 Hence,
 $$\sqrt{\eta_j(\lambda)}(A_j\pm C_j)\mathcal{R}_{\ell,2}^{\pm} = 
 a_j^{\frac{\ell+1}{2}} \sum_{\ell' = 0}^{\ell} (\pm 1)^{\ell'} \binom{\ell}{\ell'} (H_{n_j + 2\ell' - \ell-1} \pm H_{n_j + 2\ell' - \ell + 1}) + o(1)_{L^(\mathbb{R})}.$$
 Shifting the index of summation in the second sum and using Pascal's rule gives
 $$\sum_{\ell' = 0}^{\ell} (\pm 1)^{\ell'} \binom{\ell}{\ell'} (H_{n_j + 2\ell' - \ell-1} \pm H_{n_j + 2\ell' - \ell + 1})
 = \sum_{\ell' = 0}^{\ell+1} (\pm 1)^{\ell'} \binom{\ell+1}{\ell'} H_{n_j + 2\ell' - \ell-1},$$
 showing, as desired, that 
 $$\sqrt{\eta_j(\lambda)}(A_j\pm C_j)\mathcal{R}_{\ell,n_j,2}^{\pm} = \mathcal{R}_{\ell+1,n_j,2}^{\pm} + o(1)_{L^2(\mathbb{R})}.$$
 Arguing by induction on $\ell$ and using $\eqref{BornitudeReste}$, the lemma is proved.
\end{proof}
\begin{proof}[Proof of Proposition \ref{ContinuiteH}]
 The continuity of $\mathcal{H}_{\ell_1,\ell_2,j}$ on $\bigsqcup_{\lambda \in \widetilde{\Lambda}} g_j(\lambda) \times \{\lambda\}$ is easily established.
Indeed, if $(a_j',b_j',\lambda')$ is sufficiently close to $(a_j,b_j,\lambda)$ (depending on the values of $a_j, b_j, \eta_j(\lambda)$), the fact that 
$a_j$ and $b_j$ are integers entails $b_j' = b_j$ and $a_j' = a_j$.
The continuity of $\mathcal{H}_{\ell_1,\ell_2,j}$ at the point $(a_j,b_j,\lambda)$ follows from that of $\eta_j$ on $\widetilde{\Lambda}$.
We now turn to points belonging to the boundary of $\bigsqcup_{\lambda \in \widetilde{\Lambda}} g_j(\lambda) \times \{\lambda\}$.
As a corollary of Lemma \ref{LemmeCommutationAsymptotique}, for any $1 \leq j \leq d$ and $(\ell_1,\ell_2,b_j) \in \mathbb{N}^2 \times \mathbb{Z}$, we have 
\begin{multline*}
 \mathcal{H}_{\ell_1,\ell_2,j}(a_j,b_j,\lambda) = \left(\frac{a_j}{4}\right)^{\frac{\ell_1+\ell_2}{2}} \times\\
\sum_{\ell_1'=0}^{\ell_1} \sum_{\ell_2'=0}^{\ell_2} (-1)^{\ell_2'}\binom{\ell_1}{\ell_1'} \binom{\ell_2}{\ell_2'} 
\left(H_{n_j+\ell_1-2\ell_1'} \left| \right. H_{n_j+b_j+\ell_2-2\ell_2'} \right)_{L^2(\mathbb{R})} + o(1)
\end{multline*}
as $\eta_j(\lambda) \to 0$ with $\eta_j(\lambda) n_j \to \frac{a_j}{2} \in \mathbb{R}_+$. 
Performing the change of index $\ell_2' \leftarrow \ell_2 - \ell_2'$ and recalling that the Hermite functions form an orthonormal family of $L^2(\mathbb{R})$, we get
$$\mathcal{H}_{\ell_1,\ell_2,j}(a_j,b_j,\lambda) \longrightarrow  \left(\frac{a_j}{4} \right)^{\frac{\ell_1 + \ell_2}{2}} F_{\ell_1,\ell_2}(b_j)$$
as $\eta_j(\lambda) \to 0$ with $\eta_j(\lambda) n_j \to \frac{a_j}{2} \in \mathbb{R}_+$. 
\end{proof}
An immediate corollary of Proposition \ref{ContinuiteH} and Lemma \ref{NormalConvergenceW} is that, as
 $$(a_j,b_j,\lambda) \to (a_j,b_j,\lambda_0) \in g_j(\lambda) \times \eta_j^{-1}(\{0\}),$$
 we have
$$\widetilde{\mathcal{W}}_j(a_j,b_j,\lambda,x_j,y_j) \longrightarrow \sum_{\ell_1,\ell_2=0}^{\infty} \left(\frac{a_j}{4}\right)^{\frac{\ell_1+\ell_2}{2}} 
F_{\ell_1,\ell_2}(b_j) \frac{(iy_j)^{\ell_1}x_j^{\ell_2}}{\ell_1!\ell_2!}. $$
\section{The case of functions independant of the central variable.} 
\label{sec:IndependantCentralVariable}
The goal of this section is to define properly what the Fourier transform of a function independant of the central variable should be.
We begin by a convergence lemma for functions integrated against approximate Dirac masses around some boundary point $\lambda \in \Lambda_0$. 
 \begin{lemma}
  Let $\chi : \mathbb{R}^p \to \mathbb{R}^p$ be compactly supported and integrable, with 
  $$\int_{\mathbb{R}^p} \chi(\lambda) d\lambda = 1. $$
  Let $\lambda_0 \in \mathbb{R}^p$. 
  Let $\theta$ be in $\mathcal{C}_c(\hat{g}).$
  Then, as $\varepsilon \to 0$,
  $$I_{\varepsilon} := \int_{\mathbb{R}^n} \frac{1}{\varepsilon^p} \chi\left(\frac{\lambda-\lambda_0}{\varepsilon}\right) \theta(\hat{w})d\hat{w} \longrightarrow 
  \langle d\mu^{\lambda_0}, \theta(\cdot,\cdot,\lambda_0) \rangle. $$
  \label{LemmeDiracEnLambda} 
 \end{lemma}
 \begin{proof}
  With an obvious change of variable, we have 
  $$I_{\varepsilon} = 
  \int_{\mathbb{R}^p} \chi(\lambda) \left(\int_{g(\lambda_0 + \varepsilon \lambda)} \theta(a,b,\lambda_0 + \varepsilon \lambda) 
  d\mu^{\lambda_0 + \varepsilon \lambda}(a,b) \right) d \lambda. $$
  Since $\theta$ is continuous and compactly supported (hence, uniformly continuous), we have
$$\|\theta(\cdot,\cdot,\lambda_0 + \varepsilon\cdot) - \theta(\cdot,\cdot,\lambda_0)\|_{L^{\infty}(\hat{g})} \longrightarrow 0 $$
as $\varepsilon \to 0$.
Therefore, from the weak-$\ast$ continuity of the map $\lambda \mapsto d\mu^{\lambda}$, for any $\lambda \in \mathbb{R}^p$ there holds
$$\int_{g(\lambda_0 + \varepsilon \lambda)} \theta(a,b,\lambda_0 + \varepsilon \lambda) d\mu^{\lambda_0 + \varepsilon\lambda}(a,b) \longrightarrow 
\int_{g(\lambda_0)} \theta(a,b,\lambda_0)d\mu^{\lambda_0}(a,b) $$
as $\varepsilon \to 0$.
 Thanks to the compactness of the supports of both $\theta$ and $\chi$, we may apply the dominated convergence theorem to get that, as $\varepsilon \to 0$,
 $$I_{\varepsilon} \longrightarrow 
 \int_{\mathbb{R}^p}\chi(\lambda)\left(\int_{g(\lambda_0)} \theta(a,b,\lambda_0) d\mu^{\lambda_0}(a,b) \right) d \lambda 
 = \int_{g(\lambda_0)} \theta(a,b,\lambda_0) d\mu^{\lambda_0}(a,b). $$
 \end{proof}
 \begin{remark}
  The conclusion of the above theorem remains true if one replaces the compactness of the support of $\theta$ by the assumptions 
  $$\sup_{\lambda \in \mathbb{R}^p} \int_{g(\lambda)} |\theta(a,b,\lambda)| d\mu^{\lambda}(a,b) < \infty $$
  and
  $$\limsup_{R \to \infty} \sup_{\lambda \in \mathbb{R}^p} \int_{g(\lambda)} |\theta(a,b,\lambda)|\mathds{1}_{\{|a|^2 + |b|^2 \geq R^2\}} d\mu^{\lambda}(a,b) = 0. $$
  An example of admissible function is given for $\alpha > d$ by
  $$\theta_{\alpha} : (a,b,\lambda) \mapsto (1+|a|^2 + |b|^2)^{-\alpha}. $$
 \end{remark}
For $f\in L^1(\mathbb{R}^{2d})$ and $(a,b,\lambda) \in \hat{g}$, we define
$$\mathcal{G}_g^{\lambda}(f)(a,b) = \int_{\mathbb{R}^{2d}} \overline{\mathcal{W}}(a,b,\lambda,Z)f(Z) dZ. $$
\begin{theorem}
Let $\lambda_0 \in \Lambda_0$. 
Let $\chi\in \mathcal{S}(\mathbb{R}^p)$ be such that $\chi(0) = 1$ and $\mathcal{F}_{\mathbb{R}^p}(\chi)$ is compactly supported.
Then, for any $f\in L^1(\mathbb{R}^{2d})$, we have
\begin{equation}
 \mathcal{F}_g(f \otimes e^{i \langle \lambda_0, \cdot\rangle} \chi(\varepsilon \cdot)) \rightharpoonup^* (2\pi)^p \mathcal{G}_g^{\lambda_0}(f) d\mu^{\lambda_0}
 \label{EquationConvergenceSansVariableVerticale}
\end{equation}
as $\varepsilon \to 0$, in the weak sense of measures.
 \label{ConvergenceSansVariableVerticale}
\end{theorem}
 \begin{proof}
  Let $\theta \in \mathcal{C}_c(\hat{g})$.
  By definition of the Fourier transform, we have
  $$\int_{\hat{g}} \mathcal{F}_g(f\otimes e^{i\langle \lambda_0, \cdot\rangle} \chi(\varepsilon \cdot))(\hat{w}) \theta(\hat{w}) d\hat{w}
  = \int_{\mathbb{R}^p} \varepsilon^{-p} \hat{\chi}(\varepsilon^{-1}(\lambda-\lambda_0))
  \left( \int_{g(\lambda)} (G\theta)(a,b,\lambda) d\mu^{\lambda}(a,b)\right) d\lambda,$$
  where we abbreviated $\mathcal{F}_{\mathbb{R}^p}(\chi)$ into $\hat{\chi}$ for the sake of readability and defined the function $G$ by
  $$G(a,b,\lambda) := \int_{\mathbb{R}^n} f(Z) \overline{\mathcal{W}}(a,b,\lambda,Z) dZ. $$
  The assumptions on $\chi$ entail in particular that 
  $$\int_{\mathbb{R}^p}\hat{\chi}(\lambda) d\lambda = (2\pi)^p. $$
  Since $f$ belongs to $L^1(\mathbb{R}^{2d})$, $G$ is continuous and hence, the product function $G\theta$ lies in $\mathcal{C}_c(\hat{g})$.
  Applying Lemma \ref{LemmeDiracEnLambda} to $G\theta$ yields, as $\varepsilon \to 0$,
  $$\int_{\hat{g}} \mathcal{F}_g(f\otimes e^{i\langle \lambda_0, \cdot\rangle} \chi(\varepsilon \cdot))(\hat{w}) \theta(\hat{w}) d\hat{w}
  \longrightarrow (2\pi)^p \int_{g(\lambda_0)} G(a,b,\lambda_0) \theta(a,b,\lambda_0) d\mu^{\lambda_0}(a,b),$$
  It only remains to notice that, by definition, 
  $$G(a,b,\lambda_0) = \mathcal{G}_g^{\lambda_0}(f)(a,b). $$
 \end{proof}

\section{Computing the kernel at the boundary.}
\label{sec:ComputingKernelBoundary}
\subsection{Preliminary identities}
Given $(a_j,x_j,y_j,b_j) \in \mathbb{R}_+ \times \mathbb{R}^2 \times \mathbb{Z}$, we let
$$\mathcal{K}(a_j,x_j,y_j,b_j) := \sum_{\ell_1,\ell_2=0}^{\infty} \left(\frac{a_j}{4}\right)^{\frac{\ell_1+\ell_2}{2}} 
F_{\ell_1,\ell_2}(b_j) \frac{(iy_j)^{\ell_1}x_j^{\ell_2}}{\ell_1!\ell_2!}. $$
Our aim is to find a closed form for the above sum.
To  this aim, we list a few identities satisfied by this function, which will eventually help us in computing an integral form for $\mathcal{K}$.
\begin{proposition}
 For $a_j \in \mathbb{R}_+$, $x_j,y_j,x'_j,y'_j, \in \mathbb{R}$ and $b_j \in \mathbb{Z}$, there holds
 \begin{enumerate}
  \item $\mathcal{K}(0,x_j,y_j,b_j) = \delta_{0,b_j}$,
  \item $\mathcal{K}(a_j,x_j,-y_j,b_j) = \overline{\mathcal{K}(a_j,x_j,y_j,b_j)},$
  \item $\mathcal{K}(a_j,-x_j,y_j,-b_j) = \mathcal{K}(a_j,x_j,y_j,b_j),$
  \item $\mathcal{K}(a_j,x_j,-y_j,-b_j) = (-1)^{b_j}\mathcal{K}(a_j,x_j,y_j,b_j),$
  \item $-\Delta_{x_j,y_j}\mathcal{K}(a_j,x_j,y_j,b_j) = a_j\mathcal{K}(a_j,x_j,y_j,b_j),$
  \item $b_j \mathcal{K}(a_j,x_j,y_j,b_j) =  i (x_j \partial_{y_j} - y_j \partial_{x_j})\mathcal{K}(a_j,x_j,y_j,b_j),$
  \item $$\mathcal{K}(a_j,x_j+x'_j,y_j+y'_j,b_j)= \\
 \sum_{b'_j \in \mathbb{Z}} \mathcal{K}(a_j,x_j,y_j,b_j-b'_j)\mathcal{K}(a_j,x'_j,y'_j,b'_j).$$
  \item $$(|x_j|^2+|y_j|^2)\mathcal{K}(a_j,x_j,y_j,b_j) + a_j\partial_{a_ja_j}\mathcal{K}(a_j,x_j,y_j,b_j) + 2 \partial_{a_j}\mathcal{K}(a_j,x_j,y_j,b_j) = 
  \frac{b_j^2}{a_j}\mathcal{K}(a_j,x_j,y_j,b_j).$$
 \end{enumerate}
 \label{IdentitiesOnK}
\end{proposition}
We begin by proving the easiest ones and postpone the last two.
\begin{proof}[Proof of Identities $(1)-(4)$]
Identity $(1)$ stems from 
$$F_{0,0}(b_j) = \delta_{0,b_j}, $$
which is obvious.
Identity $(2)$ follows directly from the definition.
Thanks to the relation
$$F_{\ell_1,\ell_2}(-b_j) = (-1)^{\ell_2}F_{\ell_1,\ell_2}(b_j), $$
Identity $(3)$ ensues.
Finally, since $F_{\ell_1,\ell_2}(b_j) = 0$ for $b_j + \ell_1 + \ell_2$ odd, we also have
$$F_{\ell_1,\ell_2}(-b_j) = (-1)^{b_j+\ell_1}F_{\ell_1,\ell_2}(b_j), $$
which in turn implies Identity $(4)$.
\end{proof}
\begin{proof}[Proof of Identities $(5)$ and $(6)$]
 To prove Identity $(5)$, notice that as a consequence of Equation $\eqref{DecayLaplacien}$, for $(\hat{w},w) \in \widetilde{g} \times \mathbb{R}^n$,
 $$(-2n_j+1)\eta_j(\lambda)\Theta(\hat{w},w) 
 = e^{i\langle \lambda,s\rangle}\left( \left(\partial_{x_j} + \frac 12 \eta_j(\lambda)y_j \right)^2 +  \left(\partial_{y_j} - \frac 12 \eta_j(\lambda)x_j \right)^2\right) 
 \mathcal{W}(\hat{w},x,y).$$
 Hence, simplifying the complex exponentials gives
  $$(-2n_j+1)\eta_j(\lambda)\mathcal{W}(\hat{w},x,y) 
 = \left( \left(\partial_{x_j} + \frac 12 \eta_j(\lambda)y_j \right)^2 +  \left(\partial_{y_j} - \frac 12 \eta_j(\lambda)x_j \right)^2\right) 
 \mathcal{W}(\hat{w},x,y).$$
 Using the fact that $\mathcal{W}$ is a tensor product to look only at $\widetilde{\mathcal{W}}_j$, we get Identity $(5)$ in the limit 
  $$(a_j,b_j,\lambda) \to (a_j,b_j,\lambda_0) \in \mathbb{R}_+ \times \mathbb{Z} \times \eta_j^{-1}(\{0\}).$$ 
  Finally, passing to the limit in Equation $\eqref{DecayEnk}$ yields
\begin{equation}
b_j \mathcal{K}(a_j,x_j,y_j,b_j) =  i (x_j \partial_{y_j} - y_j \partial_{x_j})\mathcal{K}(a_j,x_j,y_j,b_j),
 \label{RelationRotationK}
\end{equation}
which is exactly Identity $(6)$.
\end{proof}
Before proving Identity $(7)$, we state and prove an analogue lemma for the function $\mathcal{W}$.
\begin{lemma}
For any $Z,Z' \in \mathbb{R}^{2d}$, $\lambda \in \widetilde{\Lambda}$, $(a,b) \in g(\lambda)$, the following convolution property holds.
 $$e^{-\frac i2 \langle \lambda, \sigma(Z,Z')\rangle}\mathcal{W}(a,b,\lambda,Z+Z')
= \sum_{b' \in \mathbb{Z}^d}\mathcal{W}(a-\eta(\lambda)\cdot b',b-b',\lambda,Z)\mathcal{W}(a +\eta(\lambda)\cdot(b-b'),b',\lambda,Z').$$
\label{ConvolutionW}
\end{lemma}
\begin{proof}
Let $f_1,f_2$ be in $\mathcal{S}(\mathbb{R}^{2d})$, let $\alpha$ be in $\mathcal{S}(\mathbb{R}^p)$.
For $(Z,s) \in \mathbb{R}^n$, the definition of the convolution product gives
$$((f_1 \otimes \alpha) \star (f_2 \otimes \alpha))(Z,s) = \int_{\mathbb{R}^n} f_1(Z-Z')f_2(Z') \alpha\left(s-s'-\frac 12 \sigma(Z',Z)\right)\alpha(s') dZ' ds'. $$
Taking the usual Fourier transform with respect to the central variable $s$ gives
$$\mathcal{F}_{\mathbb{R}^p}((f_1 \otimes \alpha) \star (f_2 \otimes \alpha))(Z,\lambda) = 
\hat{\alpha}(\lambda)^2\int_{\mathbb{R}^{2d}} e^{\frac i2 \langle \lambda, \sigma(Z,Z')\rangle} f_1(Z-Z')f_2(Z') dZ'. $$
Now, integrating against the function $\overline{\mathcal{W}}(\hat{w},\cdot)$, we get
$$\mathcal{F}_g((f_1 \otimes \alpha) \star (f_2 \otimes \alpha))(\hat{w}) = 
\hat{\alpha}(\lambda)^2\int_{\mathbb{R}^{2d} \times \mathbb{R}^{2d}} e^{\frac i2 \langle \lambda, \sigma(Z,Z')\rangle}\overline{\mathcal{W}}(\hat{w},Z) f_1(Z-Z')f_2(Z') dZ' dZ. $$ 
Since $\sigma$ is antisymmetric, a simple change of variable gives
\begin{equation}
 \mathcal{F}_g((f_1 \otimes \alpha) \star (f_2 \otimes \alpha))(\hat{w}) = 
\hat{\alpha}(\lambda)^2\int_{\mathbb{R}^{2d} \times \mathbb{R}^{2d}} e^{\frac i2 \langle \lambda, \sigma(Z,Z')\rangle}\overline{\mathcal{W}}(\hat{w},Z+Z') f_1(Z)f_2(Z') dZ' dZ. 
\label{ConvolutionPuisFourier}
\end{equation}
Let us now compute the Fourier transform in a different way.
Applying formula $\eqref{FormuleConvolution}$ in the Appendix, we have
$$\mathcal{F}_g((f_1 \otimes \alpha) \star (f_2 \otimes \alpha))(\hat{w}) = (\mathcal{F}_g(f_1 \otimes \alpha) \cdot \mathcal{F}_g(f_2 \otimes \alpha))(\hat{w}). $$
Expanding out the operator product and parametrizing it by $\widetilde{g}_E$, we get, for $(n,m,\lambda) \in \mathbb{N}^{2d} \times \widetilde{\Lambda}$,
$$(\mathcal{F}_g(f_1 \otimes \alpha) \cdot \mathcal{F}_g(f_2 \otimes \alpha))(n,m,\lambda) = 
\sum_{\ell \in \mathbb{N}^d} \mathcal{F}_g(f_1 \otimes \alpha)(n,\ell,\lambda) \mathcal{F}_g(f_2 \otimes \alpha)(\ell,m,\lambda). $$
Extending the definition of the Fourier transform by setting 
$$\mathcal{F}_g(f_1 \otimes \alpha)(n,\ell,\lambda) = \mathcal{F}_g(f_2 \otimes \alpha)(\ell,m,\lambda) = 0$$
whenever a component of $\ell$ is strictly negative, we also have
$$(\mathcal{F}_g(f_1 \otimes \alpha) \cdot \mathcal{F}_g(f_2 \otimes \alpha))(n,m,\lambda) = 
\sum_{\ell \in \mathbb{Z}^d} \mathcal{F}_g(f_1 \otimes \alpha)(n,\ell,\lambda) \mathcal{F}_g(f_2 \otimes \alpha)(\ell,m,\lambda). $$
Reverting to a parametrization by $\widetilde{g}$, we get, for $(a,b,\lambda) \in \widetilde{g}$,
\begin{multline*}
 (\mathcal{F}_g(f_1 \otimes \alpha) \cdot \mathcal{F}_g(f_2 \otimes \alpha))(a,b,\lambda) = \\
\sum_{b' \in \mathbb{Z}^d} \mathcal{F}_g(f_1 \otimes \alpha)(a-\eta(\lambda)\cdot b',b-b',\lambda) \mathcal{F}_g(f_2 \otimes \alpha)(a+\eta(\lambda)\cdot (b-b'),b',\lambda). 
\end{multline*}
Now, by definition of the Fourier transform $\mathcal{F}_g$, the general term of the above sum equals
$$\hat{\alpha}(\lambda)^2 \int_{\mathbb{R}^{2d} \times \mathbb{R}^{2d}} 
\overline{\mathcal{W}}(a-\eta(\lambda)\cdot b',b-b',\lambda,Z)
\overline{\mathcal{W}}(a+\eta(\lambda)\cdot (b-b'),b',\lambda,Z') f_1(Z)f_2(Z') dZ dZ'.$$
Furthermore, since $g_1$ and $g_2$ lie in $\mathcal{S}(\mathbb{R}^{2d})$, we may exchange the integral on $\mathbb{R}^{2d} \times \mathbb{R}^{2d}$ 
with the sum on $\mathbb{Z}^d$ to get
\begin{multline}
 (\mathcal{F}_g(f_1 \otimes \alpha) \cdot \mathcal{F}_g(f_2 \otimes \alpha))(a,b,\lambda) = \\
\hat{\alpha}(\lambda)^2 \int_{\mathbb{R}^{2d} \times \mathbb{R}^{2d}}
\sum_{\ell \in \mathbb{Z}^d} 
\overline{\mathcal{W}}(a-\eta(\lambda)\cdot b',b-b',\lambda,Z)
\overline{\mathcal{W}}(a+\eta(\lambda)\cdot (b-b'),b',\lambda,Z') f_1(Z)f_2(Z') dZ dZ'.
\label{FourierPuisConvolution}
\end{multline}
Having equality between the right-hand sides of $\eqref{ConvolutionPuisFourier}$ and $\eqref{FourierPuisConvolution}$ 
for any $f_1,f_2,\alpha$ in their respective Schwartz classes, the lemma follows, up to a complex conjugation on both sides.
\end{proof}
We may now prove Identity $(7)$.
\begin{proof}[Proof of Identity $(7)$]
 Since $\mathcal{W}$ writes as a tensor product in the basis $(x_1,\dots,x_d,y_1,\dots,y_d)$, Lemma \ref{ConvolutionW} gives, looking at the variables 
 $(a_j,b_j,x_j,y_j) \in \mathbb{R}_+ \times \mathbb{Z} \times \mathbb{R}^2$,
\begin{multline*}
 e^{-\frac i2 \eta_j(\lambda)x_jy'_j}\mathcal{W}_j(a_j,b_j,\lambda,x_j+x'_j,y_j+y'_j)= \\
 \sum_{b'_j \in \mathbb{Z}} \mathcal{W}_j(a_j-\eta_j(\lambda)b'_j,b_j-b'_j,\lambda,x_j,y_j)\mathcal{W}_j(a_j+\eta_j(\lambda)(b_j-b'_j),b'_j,\lambda,x'_j,y'_j). 
\end{multline*}
Then, if $\lambda \to \lambda_0 \in \eta_j^{-1}(\{0\})$, we obtain as desired
$$\mathcal{K}(a_j,x_j+x'_j,y_j+y'_j,b_j)= \\
 \sum_{b'_j \in \mathbb{Z}} \mathcal{K}(a_j,x_j,y_j,b_j-b'_j)\mathcal{K}(a_j,x'_j,y'_j,b'_j).$$
 \end{proof}
 Identity $(8)$ is much more intricate to prove and requires several intermediate steps.
 We begin with a lemma describing the effect of a multiplication operator on the kernel $\mathcal{W}$.
 To this end, define the operator $\hat{\Delta}_j$ as follows.
 For $\theta \in \mathcal{C}(\widetilde{g})$ and $\hat{w}_j = (a_j,b_j,\lambda) \in g_j(\lambda) \times \widetilde{\Lambda}$,
 \begin{multline*}
  (- \hat{\Delta}_j\theta)(\hat{w}_j) :=  
 \eta_j(\lambda)^{-2}\left(2\left(a_j+\eta_j(\lambda)\right)\theta(\hat{w}_j) \right. \\ 
 - \left(\sqrt{a_j^2 - \eta_j(\lambda)^2b_j^2}\right)\theta(\hat{w}_j^-)
 - \left( \sqrt{(a_j+2\eta_j(\lambda))^2 - \eta_j(\lambda)^2b_j^2}\right) \theta(\hat{w}_j^+) ), 
 \end{multline*}
  where
  $$\hat{w}_j^{\pm} := (a_j \pm 2\eta_j(\lambda), b_j,\lambda). $$
 \begin{lemma}
  For $\hat{w}_j = (a_j,b_j,\lambda) \in g_j(\lambda) \times \widetilde{\Lambda}$ and $Z_j = (x_j,y_j) \in \mathbb{R}^2$, we have
  $$|Z_j|^2\mathcal{W}_j(\hat{w}_j,Z_j) =  (-\hat{\Delta}_j\mathcal{W}_j)(\hat{w}_j,Z_j). $$
  \label{MultiplicationDeltaJW}
 \end{lemma}
\begin{proof}
We temporarily parametrize the space $\widetilde{g}$ by $\widetilde{g}_E$.
Let us denote 
$$n_j := \frac{\eta_j(\lambda)^{-1}a_j - b_j}{2},$$
$$m_j := \frac{\eta_j(\lambda)^{-1}a_j + b_j}{2}.$$
 From the obvious identity
 $$y_j^2 e^{i \eta_j(\lambda)\xi_jy_j} = - \eta_j(\lambda)^{-2} \left( e^{i \eta_j(\lambda)y_j \cdot}\right)'' (\xi_j) $$
 and integration by parts, we first have
 \begin{multline*}
   |Z_j|^2\mathcal{W}_j(\hat{w}_j,Z_j) = \\ \eta_j(\lambda)^{-2} 
 \int_{\mathbb{R}} e^{i \eta_j(\lambda) \xi_j y_j} \left(- \frac{d^2}{d\xi_j^2} + \eta_j(\lambda)^2 x_j^2 \right) 
 \left(H_{m_j,\eta_j(\lambda)}\left(\xi_j + \frac{x_j}{2}\right)H_{n_j,\eta_j(\lambda)}\left(\xi_j - \frac{x_j}{2}\right) \right) d\xi_j.
 \end{multline*}
Writing 
$$x_j^2 = \left(\xi_j + \frac{x_j}{2}\right)^2 + \left(\xi_j - \frac{x_j}{2}\right)^2 - 2 \left(\xi_j + \frac{x_j}{2}\right)\left(\xi_j + \frac{x_j}{2}\right),$$
we get, thanks to Equation $\eqref{RescaledHarmonicOscillator}$,
\begin{align*}
  \left(- \frac{d^2}{d\xi_j^2} + \eta_j(\lambda)^2 x_j^2 \right) 
 & \left(H_{m_j,\eta_j(\lambda)}\left(\xi_j + \frac{x_j}{2}\right)H_{n_j,\eta_j(\lambda)}\left(\xi_j - \frac{x_j}{2}\right) \right) \\
 & = 2 \eta_j(\lambda)(n_j+m_j+1)H_{m_j,\eta_j(\lambda)}\left(\xi_j + \frac{x_j}{2}\right)H_{n_j,\eta_j(\lambda)}\left(\xi_j - \frac{x_j}{2}\right) \\
 & - 2 \left( M_jH_{m_j,\eta_j(\lambda)}\right)\left(\xi_j + \frac{x_j}{2}\right)\left(M_jH_{n_j,\eta_j(\lambda)}\right)\left(\xi_j - \frac{x_j}{2}\right) \\
 & - 2\left( H_{m_j,\eta_j(\lambda)}\right)'\left(\xi_j + \frac{x_j}{2}\right)\left(H_{n_j,\eta_j(\lambda)}\right)'\left(\xi_j - \frac{x_j}{2}\right).
\end{align*}
Using Equations $\eqref{MultiplicationHermite}$ and $\eqref{DerivationHermite}$ yields
\begin{align*}
\eta_j(\lambda)^{-1}&\left(- \frac{d^2}{d\xi_j^2} + \eta_j(\lambda)^2 x_j^2 \right)
 \left(H_{m_j,\eta_j(\lambda)}\left(\xi_j + \frac{x_j}{2}\right)H_{n_j,\eta_j(\lambda)}\left(\xi_j - \frac{x_j}{2}\right) \right) \\
 & = 2 (n_j+m_j+1)H_{m_j,\eta_j(\lambda)}\left(\xi_j + \frac{x_j}{2}\right)H_{n_j,\eta_j(\lambda)}\left(\xi_j - \frac{x_j}{2}\right) \\
 & -2 \sqrt{n_jm_j} H_{m_j-1,\eta_j(\lambda)}\left(\xi_j + \frac{x_j}{2}\right)H_{n_j-1,\eta_j(\lambda)}\left(\xi_j - \frac{x_j}{2}\right) \\
 & -2 \sqrt{(n_j+1)(m_j+1)} H_{m_j+1,\eta_j(\lambda)}\left(\xi_j + \frac{x_j}{2}\right)H_{n_j+1,\eta_j(\lambda)}\left(\xi_j - \frac{x_j}{2}\right).
\end{align*}
The result follows by reverting to the variables $(a_j,b_j,\lambda)$.
\end{proof}
The next lemma describes the behaviour of the operator $- \hat{\Delta}$ as $\eta_j(\lambda) \to 0$.
 \begin{lemma}
 Let $\theta : \mathbb{R}_+^* \times \mathbb{Z} \to \mathbb{R}$ be a $\mathcal{C}^2$ function of its first argument.
 Then, as $\lambda \to \lambda_0 \in \Lambda_0$ with $\eta_j(\lambda_0) = 0$, we have, for $(a_j,b_j) \in \mathbb{R}_+^* \times \mathbb{Z}$,
 $$(-\hat{\Delta}_j\theta)(a_j,b_j) \longrightarrow (- \hat{\Delta}_j^0\theta)(a_j,b_j) := 
 -4a_j \partial^2_{a_ja_j}\theta(a_j,b_j) - 4\partial_{a_j}\theta(a_j,b_j) + \frac{b_j^2}{a_j}\theta(a_j,b_j).$$
 \label{LemmeConvergenceDeltaChapeau}
 \end{lemma}
\begin{proof}
Since $-\hat{\Delta}_j$ looks like a finite difference operator, it seems only natural to perform Taylor expansions for $\theta$ around $\hat{w}_j$ when $\eta_j(\lambda)$ is small.
At second order in the parameter $\eta_j(\lambda)$, we have
 $$\theta(a_j+2\eta_j(\lambda),b_j) = 
 \theta(a_j,b_j) + 2\eta_j(\lambda) \partial_{a_j}\theta(a_j,b_j) + \frac 12 (2\eta_j(\lambda))^2 \partial^2_{a_ja_j} \theta(a_j,b_j)+ o(\eta_j(\lambda)^2), $$
 $$\theta(a_j-2\eta_j(\lambda),b_j) = 
 \theta(a_j,b_j) - 2\eta_j(\lambda) \partial_{a_j}\theta(a_j,b_j) + \frac 12 (2\eta_j(\lambda))^2 \partial^2_{a_ja_j} \theta(a_j,b_j)+ o(\eta_j(\lambda)^2), $$
 $$\sqrt{a_j^2 - \eta_j(\lambda)^2b_j^2} = a_j - \frac 12 \frac{\eta_j(\lambda)^2b_j^2}{a_j} + o(\eta_j(\lambda)^2),$$
 $$\sqrt{(a_j+2\eta_j(\lambda))^2 - \eta_j(\lambda)^2b_j^2}
 = a_j + 2\eta_j(\lambda) -  \frac 12 \frac{\eta_j(\lambda)^2b_j^2}{a_j} + o(\eta_j(\lambda)^2).$$
 Plugging these equalities in the definition of $-\hat{\Delta}_j$ gives the result.
\end{proof}
We now have the required tools to prove the last Identity of Proposition \ref{IdentitiesOnK}.
\begin{proof}[Proof of Identity $(8)$]
 To circumvent the difficulty of a discrete-to-continuous limit, we argue by duality.
 Let $\lambda_0 \in \Lambda_0$ with $\eta_j(\lambda_0) = 0$.
 Let $\psi : \mathbb{R}_+^* \to \mathbb{R}$ be smooth and compactly supported.
 For $b_j \in \mathbb{Z}$, $\lambda \in \widetilde{\Lambda}$ and $Z_j \in \mathbb{R}^2$, let 
 $$\mathcal{A}_j(b_j,\lambda,Z_j) := \int_{\mathbb{R}_+} \mathcal{W}_j(a_j,b_j,\lambda,Z_j)\psi(a_j) d\mu^{\lambda}_{j,b_j}(a_j).$$
 Thanks to Lemma \ref{MultiplicationDeltaJW}, we have
 $$|Z_j|^2\mathcal{A}_j(b_j,\lambda,Z_j) = \int_{\mathbb{R}_+} (-\hat{\Delta}_j\mathcal{W}_j)(a_j,b_j,\lambda,Z_j)\psi(a_j) d\mu^{\lambda}_{j,b_j}(a_j).$$
 Denoting by $^t \hat{\Delta}_j$ and $^t\hat{\Delta}^0_j$ the adjoints of $\hat{\Delta}_j$ and $\hat{\Delta}_j^0$ respectively for the $L^2$ inner product, we get
 $$|Z_j|^2\mathcal{A}_j(b_j,\lambda,Z_j) = \int_{\mathbb{R}_+} \mathcal{W}_j(a_j,b_j,\lambda,Z_j)(-^t\hat{\Delta}_j\psi)(a_j) d\mu^{\lambda}_{j,b_j}(a_j).$$
 As an immediate corollary of Lemma \ref{LemmeConvergenceDeltaChapeau}, for smooth $\theta : \mathbb{R}_+^* \to \mathbb{R}$ with compact support, we have the convergence 
 $$^t \hat{\Delta}_j\theta \to \, ^t\hat{\Delta}^0_j\theta \text{ in } \mathcal{C}^0(\mathbb{R}_+^*,\mathbb{R})$$
 as $\lambda \to \lambda_0$.
 Hence, applying this convergence to $\psi$ gives, as $\lambda \to \lambda_0$,
 $$\int_{\mathbb{R}_+} \mathcal{W}_j(a_j,b_j,\lambda,Z_j)(-^t\hat{\Delta}_j\psi)(a_j) d\mu^{\lambda}_{j,b_j}(a_j) \longrightarrow 
 \int_{\mathbb{R}_+} \mathcal{W}_j(a_j,b_j,\lambda_0,Z_j)(-^t\hat{\Delta}^0_j\psi)(a_j) da_j.$$
 By definition of $-^t\hat{\Delta}^0_j$, we get, as $\lambda \to \lambda_0$,
 $$|Z_j|^2\mathcal{A}_j(b_j,\lambda,Z_j) \longrightarrow  \int_{\mathbb{R}_+} (-\hat{\Delta}^0_j\mathcal{K})(a_j,b_j,\lambda_0,Z_j)\psi(a_j) da_j.$$
 On the other hand, recalling the definition of $\mathcal{A}_j$ yields
 $$|Z_j|^2\mathcal{A}_j(b_j,\lambda,Z_j) \longrightarrow  \int_{\mathbb{R}_+} |Z_j|^2\mathcal{K}(a_j,b_j,\lambda_0,Z_j)\psi(a_j) da_j. $$
 Since the reasoning above applies to all smooth and compactly supported $\psi : \mathbb{R}^*_+ \to \mathbb{R}$, the last Identity is proved.
\end{proof}
\subsection{Another expression for $\mathcal{K}$}
 The form of Identity $(7)$, of convolution type, motivates us to look at the Fourier synthesis of $\mathcal{K}$ in its last variable.
 For $(a,x,y,z) \in \mathbb{R}_+ \times \mathbb{R}^3$, let
 $$\widetilde{\mathcal{K}}(a,x,y,z) := \sum_{b_j \in \mathbb{Z}} \mathcal{K}(a,x,y,b_j) e^{ib_jz}. $$
 The function $\widetilde{\mathcal{K}}$ is well defined, since for any $(a,x,y)$ in a bounded set $\mathcal{B}$ and any $N \in \mathbb{N}$, we have
 $$\sup_{b_j\in \mathbb{Z}}\sup_{(a,x,y) \in \mathcal{B}} (1+|b_j|^N)|\mathcal{K}(a,x,y,b_j)| < \infty. $$
 Applying Identity $(7)$ to $\widetilde{\mathcal{K}}$ gives, for $(a,x,y,z) \in \mathbb{R}_+ \times \mathbb{R}^3$,
 $$\widetilde{\mathcal{K}}(a,x+x',y+y',z) = \widetilde{\mathcal{K}}(a,x,y,z)\widetilde{\mathcal{K}}(a,x',y',z). $$
 From the definition of $\widetilde{\mathcal{K}}$ and Identity $(1)$, we infer 
 $$\widetilde{\mathcal{K}}(a,0,z) = \sum_{b_j\in \mathbb{Z}} \delta_{0,b_j} e^{ib_jz} \equiv 1.$$
 Hence, for each $(a,z) \in \mathbb{R}_+ \times \mathbb{R}$, the function $(x,y) \mapsto \widetilde{\mathcal{K}}(a,x,y,z)$ is a non trivial, 
 continuous group morphism from $\mathbb{R}^2$ to $\mathbb{R}$.
 From the equality 
 $$\widetilde{\mathcal{K}}(a,-x,-y,z) = \overline{\widetilde{\mathcal{K}}(a,x,y,z)}, $$
 which stems from Identities $(2)-(4)$, this function is also a character of $\mathbb{R}^2$.
 Thus, there exists a function $\Phi : \mathbb{R}_+ \times \mathbb{R} \to \mathbb{R}^2$ such that, for any $(a,x,y,z) \in \mathbb{R}_+ \times \mathbb{R}^3$, we have
 $$\widetilde{\mathcal{K}}(a,x,y,z) = e^{i \langle (x,y), \Phi(a,z)\rangle}.$$
 Since $\mathcal{K}$ is smooth in all variables and rapidly decaying in $b_j$, the equality 
 $$\Phi(a,z) = \overline{\widetilde{\mathcal{K}}(a,x,y,z)} \nabla_{x,y} \widetilde{\mathcal{K}}(a,x,y,z) $$
 entails the smoothness of $\Phi$ in $(a,z)$.
 Thus, Identity $(6)$ applied to $\widetilde{\mathcal{K}}$ implies, viewing $\mathbb{R}^2$ as $\mathbb{C}$,
 $$\partial_z \Phi(a,z) = i \Phi(a,z). $$
 Solving this differential equation leads to 
 $$\Phi(a,z) = e^{iz} \Phi(a,0).  $$
 Thanks to Identity $(5)$, we get
 $$|\Phi(a,z)| = \sqrt{a}. $$
 Hence, there exists a map $\phi : \mathbb{R}_+ \to \mathbb{T}$ such that, for any $(a,z)$ in $\mathbb{R}_+ \times \mathbb{R}$,
 $$\Phi(a,z) = \sqrt{a}e^{iz} \phi(a). $$
 We now transfer the information given by Identity $(8)$ on $\widetilde{\mathcal{K}}$ to find an equation on $\phi$.
 For $(a,x,y,z) \in \mathbb{R}_+ \times \mathbb{R}^3$, we have
 $$(|x|^2+|y|^2)\widetilde{\mathcal{K}}(a,x,y,z) + 4a\partial_{aa}\widetilde{\mathcal{K}}(a,x,y,z) + 4 \partial_a\widetilde{\mathcal{K}}(a,x,y,z) +  
  \frac 1a \partial^2_{zz}\widetilde{\mathcal{K}}(a,x,y,z) = 0.$$
 To keep as few terms as possible, we divide the above equation by $\widetilde{\mathcal{K}}$ and look at the imaginary part.
 We have
 $$\Im\left(\frac{4 \partial_a\widetilde{\mathcal{K}}(a,x,y,z)}{\widetilde{\mathcal{K}}(a,x,y,z)}\right) 
 = 4\left( \frac{\langle(x,y), e^{iz}\phi(a)\rangle}{2\sqrt{a}} + \sqrt{a}\langle(x,y), e^{iz}\phi'(a)\rangle \right),$$
  $$\Im\left(\frac{4a \partial_{aa}\widetilde{\mathcal{K}}(a,x,y,z)}{\widetilde{\mathcal{K}}(a,x,y,z)}\right) 
 = 4a \left(- \frac{\langle(x,y), e^{iz}\phi(a)\rangle}{4a^{\frac 32}} + 2 \frac{\langle(x,y), e^{iz}\phi'(a)\rangle}{2\sqrt{a}}
 + \sqrt{a}\langle(x,y), e^{iz}\phi''(a)\rangle\right) ,$$
  $$\Im\left(\frac{ \partial_{zz}\widetilde{\mathcal{K}}(a,x,y,z)}{a\widetilde{\mathcal{K}}(a,x,y,z)}\right) 
 = - \frac{1}{\sqrt{a}}\langle(x,y), e^{iz}\phi(a)\rangle .$$
 Gathering and simplifying these equalities yields
 $$ a\phi''(a) + 2\phi'(a) = 0.$$
 Hence, there exist two constants $C_1,C_2 \in \mathbb{C}$ such that, for all $a > 0$,
 $$\phi(a) = C_1\ln a + C_2. $$
 As $\phi$ takes its values in unit circle, it is in particular bounded, which forces $C_1$ to vanish.
 Hence, $\phi$ is actually constant and there exists $z_0 \in \mathbb{R}$ such that $\phi(a) \equiv e^{iz_0}$ for all $a \in \mathbb{R}_+$.
 To compute the value of $z_0$, we recall that Identity $(2)$ implies that for all $(a,x,y,z) \in \mathbb{R}_+ \times \mathbb{R}^3$,
 $$\widetilde{\mathcal{K}}(a,x,-y,z) = \overline{\widetilde{\mathcal{K}}(a,x,y,-z)}. $$
 Hence, for all $(a,x,y,z) \in \mathbb{R}_+ \times \mathbb{R}^3$,
 $$e^{\sqrt{a}i(x \cos(z+z_0) - y\sin(z+z_0))} = e^{-\sqrt{a}i(x \cos(-z+z_0) + y\sin(-z+z_0))} = e^{-\sqrt{a}i(x \cos(z-z_0) - y\sin(z-z_0))}. $$
 This is only possible if 
 $$z_0 \equiv \frac{\pi}{2} [\pi]. $$
 Thus, there exists $\delta \in \{\pm1\}$ such that for all $(a,x,y,z) \in \mathbb{R}_+ \times \mathbb{R}^3$,
 $$\widetilde{\mathcal{K}}(a,x,y,z) = e^{\delta \sqrt{a}i(x \sin z - y\cos z)}  . $$
 Finally, using the definition of $\widetilde{\mathcal{K}}$ for small $y > 0$ and $z = x = 0$ gives
 $$\widetilde{\mathcal{K}}(a,0,y,0) = \sum_{b_j \in \mathbb{Z}} \mathcal{K}(a,0,y,b_j) 
 = \sum_{b_j \in \mathbb{Z}} \sum_{\ell_1 \in \mathbb{N}} \left( \frac{a}{4}\right)^{\frac{\ell_1}{2}}F_{\ell_1,0}(b_j) \frac{(iy)^{\ell_1}}{\ell_1!}
 = 1 + \sqrt{a}iy + \mathcal{O}(y^2).$$
 On the other hand, the form of $\widetilde{\mathcal{K}}$ entails, again for $y > 0$ small,
 $$\widetilde{\mathcal{K}}(a,0,y,0) = e^{-\delta \sqrt{a}iy} = 1 - \delta \sqrt{a} iy + \mathcal{O}(y^2) $$
 and $\delta = -1$.
 Owing to Fourier inversion for periodic functions on the real line, we have, for all $(a,x,y,b_j) \in \mathbb{R}_+ \times \mathbb{R}^2 \times \mathbb{Z}$,
 $$\mathcal{K}(a,x,y,b_j) = \frac{1}{2\pi} \int_{-\pi}^{\pi} e^{-\sqrt{a}i(x \sin z - y\cos z)} e^{ib_jz} dz.$$

 \appendix
 
 \section{Standard computations on the Hermite functions.}
 \label{appendix:Hermite}
 In this appendix, we recall the definition of the Hermite functions along with their most useful properties.
 The computations may be found e.g. in \cite{Olver}.
 For $x\in \mathbb{R}^d$, the first Hermite function $H_0$ is defined by 
 $$H_0(x) := \pi^{-\frac d4} e^{-\frac{|x|^2}{2}}.$$
 Let $M_j$ be the multiplication operator with respect to the $j$-th variable, defined for $f : \mathbb{R} \to \mathbb{R}$ and $x\in \mathbb{R}^d$ by
 $$(M_jf)(x) := x_jf(x). $$
 Defining the creation operator $C_j$ by
 $$C_j := -\partial_j + M_j, $$
 the Hermite functions family is defined, for $n = (n_1,\dots,n_d) \in \mathbb{N}^d$n by
 $$H_n := \frac{1}{\sqrt{2^{|n|} n!}} C^n H_0, $$
 where, as usual, 
 $$C^n := \prod_{j=1}^dC_j^{n_j} $$
 and
 $$|n| := \sum_{j=1}^d n_j \: , \: n! := \prod_{j=1}^d n_j!. $$
 We also define the annihilation operator $A_j$ by
 $$A_j := \partial_j + M_j $$
 and notice that $A_j$ is the (formal) adjoint of $C_j$ for the usual inner product on $L^2(\mathbb{R}^d)$.
 It is a standard fact that the family $(H_n)_{n\in \mathbb{N}^d}$ is an orthonormal basis of $L^2(\mathbb{R}^d)$ and in particular that for any $(n,m) \in \mathbb{N}^{2d}$,
 $$\left(H_n | H_m \right)_{L^2(\mathbb{R}^d)} := \int_{\mathbb{R}^d} H_n(x) H_m(x) dx 
 =\left\{
\begin{array}{rcl}
 1 & \text{ if } n = m,\\
 0 & \text{ otherwise.}
\end{array}
\right. $$
Furthermore, the very definition of the Hermite functions entails that for any $n \in \mathbb{N}^d$ and $1 \leq j \leq d$, there holds
\begin{equation}
 C_j H_n = \sqrt{2(n_j+1)}H_{n+\delta_j} 
 \label{CreationHermite}
\end{equation}
and by duality, we get
\begin{equation}
 A_j H_n = \sqrt{2n_j}H_{n-\delta_j},
 \label{AnnihilationHermite}
\end{equation}
where $n \pm \delta_j := (n_1,\dots,n_j \pm 1, \dots, n_d)$.
Adding and subtracting these two equalities gives, for $n \in \mathbb{N}^d$ and $1 \leq j \leq d$,
\begin{equation}
M_j H_n = \frac{\sqrt{2}}{2} (\sqrt{n_j}H_{n-\delta_j} + \sqrt{n_j+1}H_{n-\delta_j}) 
 \label{MultiplicationHermite}
\end{equation}
\begin{equation}
A_j H_n = \frac{\sqrt{2}}{2} (\sqrt{n_j}H_{n-\delta_j} - \sqrt{n_j+1}H_{n-\delta_j}). 
 \label{DerivationHermite}
\end{equation}
Also, combining the action of $C_j$ and $A_j$ gives, for $n \in \mathbb{N}^d$ and $1 \leq j \leq d$,
$$-\Delta_{\text{osc},j}H_n := (-\partial_j^2 + M_j^2) H_n = (C_jA_j + \text{Id})H_n = (2n_j+1)H_n. $$
Now, if $\eta = (\eta_1,\dots,\eta_d) \in (\mathbb{R}_+^*)^d$ and $n \in \mathbb{N}^d$, we define the rescaled Hermite function $H_{n,\eta}$ by
$$H_{n,\eta} := |\eta|^{\frac d4} H_n(|\eta|^{\frac d4} \cdot). $$
These functions satisfy identities similar to those of the usual Hermite functions.
In particular, they also form an orthonormal basis of $L^2(\mathbb{R}^d)$ and for $n\in \mathbb{N}^d$, $\eta = (\eta_1,\dots,\eta_d) \in (\mathbb{R}_+^*)^d$, we have
\begin{equation}
 (-\partial_j^2 + \eta_j^2M_j^2) H_{n,\eta} = \eta_j (2n_j+1)H_{n,\eta}. 
 \label{RescaledHarmonicOscillator}
\end{equation}
We also state and prove here a technical lemma on the growth of the $L^2$ norms of Hermite functions to which multiple derivatives or multiplication operators are applied.
\begin{lemma}
For $n,\ell \in\mathbb{N}$, we have
$$\|H_n^{(\ell)}\|_{L^2(\mathbb{R})} \leq (2n+2\ell)^{\frac{\ell}{2}} $$
and
$$\|M^{\ell}H_n\|_{L^2(\mathbb{R})} \leq (2n+2\ell)^{\frac{\ell}{2}}. $$
 \label{RecurrenceChianteHermite}
\end{lemma}
\begin{proof}
 The proof is a simple induction on $\ell \in \mathbb{N}$.
 For $\ell = 0$ both inequalities are obvious, for the Hermite functions are an orthonormal family of $L^2(\mathbb{R})$.
 Given $\ell \in \mathbb{N}$, assume the inequalities to be true for $\ell$ and all $n \in \mathbb{N}$.
 Owing to Equation $\eqref{MultiplicationHermite}$ and the induction assumption, we have, for $n\in\mathbb{N}$,
 \begin{align*}
   \|M^{\ell+1}H_n\|_{L^2(\mathbb{R})} &\leq \frac{1}{\sqrt{2}}  \left(\sqrt{n} \|M^{\ell}H_{n-1}\|_{L^2(\mathbb{R})} + \sqrt{n+1} \|M^{\ell}H_{n+1}\|_{L^2(\mathbb{R})} \right) \\
   & \leq \frac{1}{\sqrt{2}} \left(\sqrt{n} (2n+2\ell-2)^{\frac{\ell}{2}} + \sqrt{n+1}  (2n+2\ell+2)^{\frac{\ell}{2}} \right) \\
   & \leq \frac{1}{\sqrt{2}} \left(\sqrt{n+\ell+1} (2n+2\ell+2)^{\frac{\ell}{2}} + \sqrt{n+\ell+1}  (2n+2\ell+2)^{\frac{\ell}{2}} \right) \\
   & = (2n+2(\ell+1))^{\frac{\ell+1}{2}}.
 \end{align*}
Thr proof for multiple applications of the derivative is similar.

\end{proof}

 \section{The representation-theoretic Fourier transform.}
 \label{appendix:FourierRepresentation}
 We collect here some standard results about the Fourier transform as defined through unitary irreducible reprentations.
 We refer the reader to \cite{BahouriFermanianGallagherHeisenberg}, \cite{BahouriFermanianGallagher2Step}, \cite{BealsGreiner}, \cite{CorwinGreenleaf},
 \cite{FarautHarzallah}, \cite{FischerRuzhansky}, \cite{Folland}, \cite{Geller}, \cite{LavanyaThangavelu}, \cite{Rudin}, \cite{Stein}, \cite{Taylor} and \cite{Thangavelu}
 for further details.
 We begin with a familiar continuity statement on $L^1(\mathbb{R}^d)$.
 \begin{theorem}
 The Fourier transformation is continuous in all its variables, in the following sense.
 \begin{itemize}
  \item  For any $\lambda \in \widetilde{\Lambda}$ and $\nu \in \mathbb{R}^t$, the map
 $$\mathcal{F}^g(\cdot)(\lambda,\nu) : L^1(\mathbb{R}^d) \longrightarrow \mathcal{L}(L^2(\mathbb{R}^d))$$
 is linear and continuous, with norm bounded by $1$.
 \item For any $u \in L^2(\mathbb{R}^d)$ and $f\in L^1(\mathbb{R}^d)$, the map
 $$\mathcal{F}^g(f)(\cdot,\cdot)(u) : \widetilde{\Lambda} \times \mathbb{R}^t \longrightarrow L^2(\mathbb{R}^d) $$
 is continuous.
 \end{itemize}
 \end{theorem}
 We continue with the analogues of the Plancherel and inversion formula in $L^2(\mathbb{R}^d)$.
 Let 
 $$\|\cdot\|_{HS(L^2(\mathbb{R}^d))}$$ 
 be the usual Hilbert-Schmidt norm on operators acting on $L^2(\mathbb{R}^d)$ and 
 $$\text{Pf}(\lambda) := \prod_{j=1}^d \eta_j(\lambda) $$
 be the (absolute value of the) pfaffian of the matrix $\mathcal{U}^{(\lambda)}$.
 With these notations, the Fourier transform $\mathcal{F}^g(\cdot)$ extends, up to a multiplicative constant, to an isometry from $L^2(\mathbb{R}^d)$ 
 to the two-parameter families $$(A(\lambda,\nu))_{(\lambda,\nu) \in \widetilde{\Lambda} \times \mathbb{R}^t}$$
 of Hilbert-Schmidt operators, endowed with the norm
 $$\|A\| := \left(\int_{\widetilde{\Lambda} \times \mathbb{R}^t} \|A(\lambda,\nu)\|_{HS(L^2(\mathbb{R}^d))}^2 \text{Pf}(\lambda) d\lambda d\nu\right)^{\frac 12}. $$
 More precisely, we have the following theorem.
 \begin{theorem}
  There exists some constant $\kappa > 0$ depending only on the choice of the group such that, for any $f \in L^2(\mathbb{R}^d)$, there holds
 $$\int_{\mathbb{R}^n} |f(w)|^2 dw = \kappa \int_{\widetilde{\Lambda} \times \mathbb{R}^t} \|\mathcal{F}^g(f)(\lambda,\nu)\|_{HS(L^2(\mathbb{R}^d))}^2 \textrm{Pf}(\lambda) d\lambda d\nu. $$
 \end{theorem}
 On the Heisenberg group $\mathbb{H}^d$, the pfaffian is simply $\text{Pf}(\lambda) = |\lambda|^d$ and the value of $\kappa$ is known, namely
 $$\kappa(\mathbb{H}^d) = \frac{2^{d-1}}{\pi^{d+1}}.$$
 In this context, the inversion formula reads, for $f\in L^1(\mathbb{R}^d)$ and almost every $w \in \mathbb{R}^n$,
 $$f(w) = \kappa \int_{\widetilde{\Lambda} \times \mathbb{R}^p} \text{tr}((u^{\lambda,\nu}_w)^* \mathcal{F}^g(f)(\lambda,\nu)) \text{Pf}(\lambda) d\lambda d\nu, $$
 with the same constant $\kappa > 0$.
 Finally, the Fourier transform exchanges as usual convolution and product, in the following sense.
 The convolution operator $\ast : L^1(\mathbb{R}^d) \times L^1(\mathbb{R}^d) \to L^1(\mathbb{R}^d)$ is defined as follows.
 For any $f_1,f_2 \in L^1(\mathbb{R}^d)$ and $(Z,s) \in \mathbb{R}^n$,
 $$(f_1 \ast f_2)(Z,s) := \int_{\mathbb{R}^n} f_1((Z,s)\cdot (-Z',-s')) f_2(Z',s') dZ' ds'.$$
 The convolution-product intertwining through the Fourier transform may now be stated.
 \begin{theorem}
   For any $f_1,f_2 \in L^1(\mathbb{R}^d)$ and $(\lambda,\nu) \in \widetilde{\Lambda} \times \mathbb{R}^t$, we have,
 denoting by $\cdot$ the operator composition on $\mathcal{L}(L^2(\mathbb{R}^d))$,
 \begin{equation}
  \mathcal{F}^g(f_1\ast f_2)(\lambda,\nu) = \mathcal{F}^g(f_1)(\lambda,\nu)\cdot \mathcal{F}^g(f_2)(\lambda,\nu).
  \label{FormuleConvolution}
 \end{equation}
 \end{theorem}
 Finally, as in the classical commutative theory, the Fourier transform allows us to diagonalize the action of the subelliptic laplacian on $g$, whose definition we recall.
 If $V := (V_1,\dots,V_m)$ is an orthonormal family such that $V \cup \{\partial_{s_k}, 1 \leq k \leq p\}$ is a basis of $g$, 
 the subelliptic laplacian with respect to the family $V$ is, by definition,
  $$\Delta_g := \sum_{j=1}^m V_j^2. $$
 One may prove that this definition is independant of $V$ provided it satisfies the two stated conditions.
 For $\lambda \in \widetilde{\Lambda}$, $f\in \mathcal{C}^{\infty}(\mathbb{R}^d)$ and $x\in\mathbb{R}^d$, define also
 the rescaled harmonic oscillator 
 $$(-\Delta_{\text{osc},\eta(\lambda)}f)(x):= (-\Delta + |\eta(\lambda)\cdot x|^2)f(x). $$
 Above, $\Delta$ is the standard laplacian acting on smooth functions on $\mathbb{R}^d$.
 We may know state how the Fourier transform diagonalizes the action of the laplacian.
 \begin{theorem}
  Let $f\in  \mathcal{C}^{\infty}_c(\mathbb{R}^n)$.
  Let $(\lambda,\nu) \in \widetilde{\Lambda} \times \mathbb{R}^t$.
  Then, for $u \in \mathcal{C}^{\infty}_c(\mathbb{R}^d)$, there holds
  $$\mathcal{F}^g(-\Delta_g f)(\lambda,\nu)(u) = \mathcal{F}^g(f)(\lambda,\nu) \left( -\Delta_{\text{osc},\eta(\lambda)}u + |\nu|^2u\right).$$
 \end{theorem}

\end{document}